\documentclass{amsart}
\usepackage{graphicx, enumerate, url}
\usepackage{amssymb,mathtools}
\usepackage{algpseudocode}
\usepackage{color}
\usepackage{tikz}
\usepackage{kbordermatrix}
\usepackage{booktabs}
\numberwithin{equation}{section}
\theoremstyle{plain}
\newtheorem{theorem}{Theorem}[section]
\newtheorem{proposition}[theorem]{Proposition}
\newtheorem{lemma}[theorem]{Lemma}
\newtheorem{corollary}[theorem]{Corollary}

\newtheorem{algorithm}[theorem]{Algorithm}

\theoremstyle{definition}

\newtheorem{example}[theorem]{Example}

\newtheorem{remark}[theorem]{Remark}

\let\S\undefined

\DeclareMathOperator{\rank}{rank}
\DeclareMathOperator{\srank}{srank}
\DeclareMathOperator{\sbrank}{sbrank}

\DeclareMathOperator{\Gr}{Gr}

\DeclareMathOperator{\vrank}{vrank}

\DeclareMathOperator{\S}{S}
\DeclareMathOperator{\I}{I}

\DeclareMathOperator{\expdim}{exp.dim}

\newcommand{\mA}{\mathcal{A}}

\newcommand{\mc}{\mathcal}
\newcommand{\re}{\mathbb{R}}
\newcommand{\PP}{\mathbb{P}}
\newcommand{\N}{\mathbb{N}}

\newcommand{\cpx}{\mathbb{C}}
\newcommand{\bbm}{\begin{bmatrix}}
\newcommand{\ebm}{\end{bmatrix}}
\newcommand{\lmd}{\lambda}
\newcommand{\bca}{\begin{cases}}
\newcommand{\eca}{\end{cases}}
\newcommand{\bit}{\begin{itemize}}
\newcommand{\eit}{\end{itemize}}
\newcommand{\af}{\alpha}
\newcommand{\bt}{\beta}
\newcommand{\reff}[1]{(\ref{#1})}
\newcommand{\bea}{\begin{eqnarray}}
\newcommand{\eea}{\end{eqnarray}}
\newcommand{\be}{\begin{equation}}
\newcommand{\ee}{\end{equation}}

\newcommand{\red}[1]{\textcolor{red}{#1}}

\begin{document}

\title[Symmetric Tensor Decompositions On Varieties]
{Symmetric Tensor Decompositions On Varieties}

\author{Jiawang~Nie}
\address{Jiawang Nie, Department of Mathematics,
University of California San Diego,
9500 Gilman Drive, La Jolla, CA, USA, 92093.}
\email{njw@math.ucsd.edu}

\author[Ke Ye]{Ke~Ye}
\address{Ke Ye and Lihong Zhi, KLMM, Academy of Mathematics and Systems Science,
Chinese Academy of Sciences,
Beijing 100190, China}
\email{keyk@amss.ac.cn, lzhi@mmrc.iss.ac.cn}

\author[Lihong Zhi]{Lihong~Zhi}

\begin{abstract}
This paper discusses the problem of symmetric tensor decomposition
on a given variety $X$: decomposing a symmetric tensor
into the sum of tensor powers of vectors contained in $X$.
In this paper, we first study geometric and algebraic properties of such decomposable tensors, which are crucial to the practical computations of such decompositions.  For a given tensor, we also develop a criterion for the existence of a symmetric decomposition on $X$. Secondly and most importantly, we propose a method for computing symmetric tensor decompositions on an arbitrary $X$. As a specific application, Vandermonde decompositions
for nonsymmetric tensors can be computed by the proposed algorithm.
\end{abstract}



\maketitle

\section{Introduction}

Let $n,d>0$ be integers and $\cpx$ be the complex field.
Denote by $\operatorname{T}^d(\cpx^{n+1})$ the space of
$(n+1)$-dimensional complex tensors of order $d$.
For $\mA \in \operatorname{T}^d(\cpx^{n+1})$ and an integral tuple
$(i_1, \ldots, i_d)$, $\mA_{i_1 \dots i_d}$ denotes the
$(i_1, \ldots, i_d)$th entry of $\mA$,  where $0\le i_1, \ldots, i_m \le n$.
The tensor $\mA$ is {\em symmetric} if
\[
\mA_{i_1 \ldots  i_d} = \mA_{j_1 \ldots j_d}
\]
whenever $(j_1, \ldots, j_d)$ is a permutation of $(i_1, \ldots, i_d)$.
Let $\operatorname{S}^d(\mathbb{C}^{n+1})$
be the subspace of all symmetric tensors in $\operatorname{T}^d(\cpx^{n+1})$.
For a vector $u$, denote by $u^{\otimes d}$ the $d$th tensor power of $u$,
i.e., $u^{\otimes d}$ is the tensor such that
$(u^{\otimes d})_{i_1,\dots, i_d} = u_{i_1} \cdots u_{i_d}$.
As shown in \cite{CGLM08}, for each
$\mA \in \operatorname{S}^d (\mathbb{C}^{n+1})$,
there exist vectors $u_1,\ldots, u_r \in \mathbb{C}^{n+1}$ such that
\be \label{A=sum:uid} %\label{A=sum:ui^d}
\mA = (u_1)^{\otimes d} + \cdots + (u_r)^{\otimes d}.
\ee
The above is called a \emph{symmetric tensor decomposition} (STD)
for $\mA$. The smallest such $r$ is called the {\it symmetric rank} of $\mA$,
for which we denote as $\srank(\mA)$.
If $\srank(\mA) =r$, $\mA$ is called a rank-$r$ tensor and
\reff{A=sum:uid} is called a symmetric rank decomposition,
which is also called a {\it Waring decomposition} in some references.
The rank of a generic symmetric tensor is given
by a formula in Alexander-Hirschowitz \cite{AlxHirs95}.
We refer to \cite{CGLM08} for symmetric tensors and their symmetric ranks,
and refer to \cite{Land12,Lim13}
for general tensors and their ranks.

This paper concerns symmetric tensor decompositions on a given set.
Let $X\subseteq \cpx^{n+1}$ be a homogeneous set,
i.e., $tx \in X$ for all $t \in \cpx$ and $x\in X$.
If each $u_j \in X$, \reff{A=sum:uid} is called a
a \emph{symmetric tensor decomposition on $X$} (STDX) for $\mA$.
The STDX problem has been studied in applications for various choices of $X$.
Symmetric tensor decompositions have broad applications
in quantum physics \cite{uemura2012symmetric}, algebraic complexity theory \cite{chiantini2017polynomials,landsberg2006border,strassen1969gaussian,YL2016},
numerical analysis \cite{LRSTA,ye2018tensor}.
More tensor applications can be found in~\cite{KolBad09}.

When $X = \cpx^{n+1}$, the STDX is just the classical
symmetric tensor decomposition,
which has been studied extensively in the literature.
Binary tensor (i.e., $n=1$)
decomposition problems were discussed in \cite{ComSei11}.
For higher dimensional tensors,
the Catalecticant type methods \cite{IarKan99}
are often used when their ranks are low.
For general symmetric tensors,
Brachat et al.~\cite{BRACHAT20101851} proposed a method
by using Hankel (and truncated Hankel) operators.
It is equivalent to computing a new tensor
whose order is higher but the rank is the same as the original one. Oeding and Ottaviani~\cite{oeding2013eigenvectors} proposed to
compute symmetric decompositions by Koszul flattening,
tensor eigenvectors and vector bundles.
Other related work on computing symmetric tensor decompositions
can be found in \cite{BalBer12,BerGimIda11}.
For generic tensors of certain ranks,
the symmetric tensor decomposition is unique.
As shown in~\cite{GalMel},
a generic $\mA \in \operatorname{S}^m(\cpx^{n+1})$
has a unique Waring decomposition if and only if
\[
(n,m,r) \in \big\{ (1,2k-1,k),\, (3,3,5),\, (2,5,7) \big\}.
\]
When $\mA \in \operatorname{S}^m(\cpx^{n+1})$
is a generic tensor of a subgeneric rank $r$
(i.e., $r$ is smaller than the value given by the
Alexander-Hirschowitz formula; see \cite{AlxHirs95,CGLM08}.)
and $m\geq 3$, the Waring decomposition is unique,
with only three exceptions \cite{ChOtVan15}.

When $X=\{(a^n, a^{n-1}b, \ldots, ab^{n-1}, b^n):\, a,b\in \cpx\}
\subseteq \cpx^{n+1}$,
i.e., $X$ is the affine cone of a rational normal curve
in the projective space $\mathbb{P}^n$, the STDX
becomes a Vandermonde decomposition for symmetric tensors.
It only exists for Hankel tensors, which were introduced
in \cite{papy2005exponential} for studying the harmonic retrieval problem.
Hankel tensors were discussed in \cite{qi2015hankel}.
Relations among various ranks of Hankel tensors were studied in \cite{nie2017hankel}.
More applications of Hankel tensors can be found in
\cite{signoretto2011kernel,trickett2013interpolation}.

When $X = \{ a_1 \otimes \cdots \otimes a_k: a_1, \ldots, a_k \in \cpx^m \}
\subseteq \cpx^{n+1}$ with $n+1 = m^k$, i.e., $X$ is a Segre variety,
the STDX becomes a Vandermonde decomposition
for nonsymmetric tensors. This has broad applications in signal processing
\cite{de1998matrix,lim2010multiarray,nion2010tensor,sun2012accurate}.
Vandermonde decompositions for nonsymmetric tensors are closely related to
secant varieties of Segre-Veronese varieties,
which has been studied vastly
\cite{abo2013dimensions,LP2013,raicu2012secant}.
In the subsection~\ref{ssc:VDnst}, we will discuss this question
with more details. Interesting, it can be shown that
every nonsymmetric tensor has a Vandermonde decomposition,
which is different from the symmetric case.

%\smallskip \noindent
%{\bf Contributions}
\subsection*{Contributions}
This paper focuses on computing symmetric tensor decompositions
on a given set $X$. We assume that $X\subseteq \mathbb{C}^{n+1}$
is a variety that is given by homogeneous polynomial equations.
Generally, symmetric tensor decompositions on $X$ can be theoretically studied
by secant varieties of the Veronese embedding of $X$
\cite{landsberg2013equations,sam2017ideals,sam2017syzygies}.
From this view, one may expect to get polynomials
which characterize the symmetric $X$-rank of a given symmetric tensor.
In this paper, we give a method for
computing symmetric $X$-rank decompositions.
It is based on the tool of generating polynomials
that were recently introduced in \cite{nie2017generating}.
The work \cite{nie2017generating} only discussed the case
$X = \cpx^{n+1}$. When $X$ is a variety, i.e.,
the method in \cite{nie2017generating} does not work,
because $u_i \in X$ is required in \reff{A=sum:uid}.
We need to modify the approach in \cite{nie2017generating},
by posing additional conditions on generating polynomials.
For this purpose, we give a unified framework
for computing symmetric tensor decompositions on $X$,
which sheds light on the study of
both theoretical and computational aspects of tensor decompositions.

The paper is organized as follows. Section~\ref{sec:preliminaries}
gives some basics for tensor decompositions.
Section~\ref{Sec:symmetric X-rank} studies some
properties of symmetric tensor decompositons on $X$.
Section~\ref{sec:GPM} defines generating polynomials and generating matrices.
It gives conditions ensuring that the computed vectors
belong to the given set in symmetric tensor decompositions.
Section~\ref{sc:frm} presents a unified framework to do the computation.
Last, Section~\ref{sec:applications} gives
various examples to show how the proposed method works.

\section{Preliminaries}\label{sec:preliminaries}

\noindent
{\bf Notation}
The symbol $\N$ (resp., $\re$, $\cpx$) denotes the set of
nonnegative integers (resp., real, complex numbers).
For $\af :=(\af_1, \ldots, \af_n) \in \N^n$, define
$|\af|: = \af_1 + \cdots + \af_n$.
For a degree $d>0$, denote the index set
\[
\mathbb{N}^n_d = \{ \af :=(\af_1, \ldots, \af_n) \in \N^n \mid
|\af|  \leq d \}.
\]
For a real number $t \in \mathbb{R}$, we denote by $\lceil t \rceil $
the smallest integer $n$ such that $n \ge t$.
Let $x:=(x_0,\dots, x_n)$ and $\cpx[x] := \cpx[x_0,\dots, x_n]$
denote the ring of polynomials in $x$ and with complex coefficients.
For a degree $m$,  $\cpx[x]_m$ denotes the subset
of polynomials whose degrees are less than or equal to $m$,
and $\cpx[x]_m^h$ denotes the subset
of forms whose degrees are equal to $m$.
The cardinality of a finite set $T$ is denoted as $|T|$.
For a finite set $\mathbb{B} \subseteq \cpx[x]$
and a vector $v \in \cpx^n$, denote
\be \label{v:to:*m}
[v]_\mathbb{B}  := \big( p(v) \big)_{p \in \mathbb{B}},
\ee
the vector of polynomials in $\mathbb{B}$ evaluated at $v$.
For a complex matrix $A$, $A^T$ denotes its transpose and $A^*$
denotes its conjugate transpose.
For a complex vector $u$, $\| u \|_2 = \sqrt{u^*u}$
denotes the standard Euclidean norm.
The $e_i$ denotes the standard $i$-th unit vector in $\N^n$.
For two square matrices $X,Y$ of the same dimension,
their commutator is $[X, Y] := XY-YX$.

\subsection{Equivalent descriptions for symmetric tensors}
\label{ssc:equdes}

There is a one-to-one correspondence between
symmetric tensors in $\S^d(\mathbb{C}^{n+1})$
and homogeneous polynomials of degree $d$ and in $(n+1)$ variables.
%A standard labelling for $\mA \in \S^d(\mathbb{C}^{n+1})$ is
%\be  \label{label:af=j}
%\mA = (\mA_{j_1,\dots, j_d}),\quad 0 \le j_1,\dots, j_d \le n.
%\ee
When $\mA \in \S^d(\mathbb{C}^{n+1})$ is symmetric,
we can equivalently use $\alpha = (\alpha_1,\dots,\alpha_n) \in \N_d^n$
to label $\mA$ in the way that
\be \label{newidx:af}
\mA_{\alpha} \coloneqq \mA_{j_1,\dots, j_d},\quad
\mbox{if} \quad x_0^{d-|\af|}x_1^{\alpha_1}\cdots x_n^{\alpha_n}
= x_{j_1} \cdots x_{j_d}.
\ee
The symmetry guarantees that the labelling $\mA_\alpha$ is well-defined.
For $\mA$, define the homogeneous polynomial
(i.e., a form) in $x := (x_0, \dots, x_n)$ and of degree $d$:
\be \label{df:F(x)}
\mA(x) \, := \, {\sum}_{j_1,\ldots, j_d=0}^n
\mA_{j_1,\dots,j_d} x_{j_1}\cdots x_{j_d}.
\ee
If $\mA$ is labeled as in \reff{newidx:af}, then
\[
\mA(x) = {\sum}_{\alpha \in \N_d^n  }
\binom{d}{d-|\af|, \alpha_1,\dots,\alpha_n}
\mA_{\alpha} x_0^{d-|\af|} x_1^{\af_1} \cdots x_n^{\af_n}
\]
where
$
\binom{d}{\alpha_0,\dots,\alpha_n} \coloneqq \frac{d!}{\alpha_0 ! \cdots \alpha_n!}.
$
The decomposition \reff{A=sum:uid} is equivalent to
\[
\mA(x) = {\sum}_{j=1}^r  \big( (u_j)_0 x_0 + \cdots + (u_j)_n x_n \big)^d.
\]
For a polynomial $ p = \sum_{ \af \in \N^n_d} p_\af
y_1^{\af_1} \cdots y_n^{\af_n}$ and $\mA \in \S^d(\cpx^{n+1})$,
define the operation
\be \label{<p,mA>}
\langle p, \mA \rangle = {\sum}_{ \af \in \N^n_d} p_\af \mA_{\af}
\ee
where $\mA$ is labeled as in \reff{newidx:af}.
For fixed $p$, $\langle p, \cdot \rangle$
is a linear functional on $\S^d(\cpx^{n+1})$,
while for fixed $\mA \in \S^d(\cpx^{n+1})$,
$\langle \cdot, \mA \rangle$ is a linear functional on
$\cpx[y_1,\ldots, y_n]_d$.

\subsection{Algebraic varieties}

A set $I \subseteq \cpx[x]:=\mathbb{C}[x_0,\dots, x_n]$
is an ideal if $I \cdot \cpx[x] \subseteq I$
and $I+I \subseteq I$. For polynomials
$f_1,\dots, f_s\in \mathbb{C}[x]$, let $\left\langle f_1,\dots,f_s \right\rangle $
denote the smallest ideal that contains $f_1,\ldots, f_s$.
%Let $(x_0,\dots, x_n)$ be the coordinate vector for $\mathbb{C}^{n+1}$.
A subset $X \subseteq \mathbb{C}^{n+1}$ is called an \emph{affine variety} if
$X$ is the set of common zeros of some polynomials in $\cpx[x]$.
The \emph{vanishing ideal} of $X$ is the ideal consisting of
all polynomials identically vanish on $X$.

Two nonzero vectors in $\mathbb{C}^{n+1}$ are \emph{equivalent}
if they are parallel to each other.
Denote by $[u]$ the set of all nonzero vectors that
are equivalent to $u$; the set $[u]$ is called the equivalent class of $u$.
The set of all equivalent classes $[u]$ with $0\ne u \in \cpx^{n+1}$
is the projective space $\PP^n$, or equivalently,
$\PP^n = \{ [u]: \, 0 \ne u \in \cpx^{n+1} \}$.
%
%We denote by $\mathbb{P}^n$ the quotient space
%$(\mathbb{C}^{n+1}\setminus \{0\} )/\sim $ where
%$(a_0,\dots,a_n) \sim (b_0,\dots, b_n)$ if and only if there exists
%$\lambda\in \mathbb{C}\setminus \{0\}$ such that
%$a_j = \lambda b_j,j=0,\dots, n$. We call $\mathbb{P}^n$
%the $n$-dimensional projective space. We also denote by $[x]$
%the equivalence class represented by $x\in \mathbb{C}^{n+1}\setminus \{0\}$.
%
A subset $Z \subseteq \mathbb{P}^n$ is said to be a \emph{projective variety}
if there are homogeneous polynomials $h_1, \ldots, h_t \in \cpx[x]$ such that
\[
Z = \{ [u] \in \PP^n:\,
h_1(u) = \cdots = h_t(u) = 0 \}.
\]
The \emph{vanishing ideal} $\mc{I}(Z)$ is defined to be the ideal
consisting of all polynomials identically vanish on $Z$.
For each nonnegative integer $m$, we denote by $\mc{I}_m (Z)$
the linear subspace of polynomials of degree $m$ in $\mc{I}(Z)$.
A projective variety $Z \subseteq \mathbb{P}^n$ is said to be
\emph{nondegenerate} if $Z$ is not contained in any proper
linear subspace of $\mathbb{P}^n$, i.e., $\mc{I}_1(Z) = \{0\}$.

In the Zariski topology for $\cpx^{n+1}$ and $\PP^n$,
the closed sets are varieties and the open sets
are complements of varieties.
For a projective variety $Z \subseteq \mathbb{P}^n$,
its Hilbert function $h_Z:\mathbb{N} \to \mathbb{N}$ is defined as
$
h_Z(d) \, \coloneqq \, \dim \mathbb{C}[Z]_d.
$
As in \cite{CLO,Harris1992,Hartshorne1977}, when
$d$ is sufficiently large, $h_Z(d)$ is a polynomial and
\[
h_Z(d) = \frac{e}{m!} d^m + O(d^{m-1}),
\quad e = \deg(Z),\quad m = \dim(Z),
\]
where $O( d^{m-1})$ denotes terms of order at most $m-1$.

For an affine variety $X \subseteq \cpx^{n+1}$,
we denote by $\PP X$ the projective set of
equivalent classes of nonzero vectors in $X$, i.e.,
$
\PP X = \{ [u]: \, 0 \ne u \in X \}.
$
If $\PP X = Z$, then $X$ is called the
\emph{affine cone} of $Z$.
Clearly, the vanishing ideal of $X$ is the same as that of $\PP X$,
i.e., $\mathcal{I}(X) = \mathcal{I}(\PP X)$.
For a degree $m>0$, the set
\[
\mc{I}_m(X) := \cpx[x]^h_m \cap  \mc{I}(X)
\]
is the space of all forms of degree $m$ vanishing on $X$.

\subsection*{Veronese maps}
For an affine variety $X \subseteq \mathbb{C}^{n+1}$, let
$v_d(X) \subseteq \S^d(\mathbb{C}^{n+1})$ be the image of $X$
under the \emph{Veronese map}:
\[
v_d: \mathbb{C}^{n+1}  \to  \S^d (\mathbb{C}^{n+1}),\quad
u \mapsto u^{\otimes d}.
\]
Note that $v_d^{-1 } (u^{\otimes d})  = \{\omega^i u:i=0,\dots,d-1\}$,
where $\omega$ is a primitive $d$-th root of $1$.
Therefore, the dimension of $v_d(X)$ is the same as that of $X$.
In particular, for
\[
C := \{(x_0, x_1, \ldots, x_n) \in \cpx^{n+1}:
x_i x_j = x_k x_l, \, \forall \, i+j = k +l \},
\]
the set $v_d(C)$ is a variety of tensors
$\mA \in \operatorname{S}^d(\mathbb{C}^2)$
that is defined by the equations
\[
\mA_\alpha \mA_\beta - \mA_\gamma \mA_\tau=0 \,\quad
(\alpha+\beta = \gamma + \tau,\,
\alpha,\beta,\gamma,\tau\in \mathbb{N}_{d}^2).
\]
In the above, the tensors in $\operatorname{S}^d(\mathbb{C}^2)$
are labelled by vectors in $\mathbb{N}_{d}^2$.
The projectivization $\PP v_d(C)$ is called the rational normal curve
in the projective space $\mathbb{P} \operatorname{S}^d(\mathbb{C}^2)\simeq \PP^d$.
For a projective variety $Z \subseteq \PP^n$,
the Veronese embedding map $v_d$ is defined in the same way as
\[
v_d: \PP^n \to  \PP \S^d (\mathbb{C}^{n+1}),\quad
[u] \mapsto [u^{\otimes d}].
\]
Note that $\PP v_d(X) = v_d( \PP X )$
for every affine variety $X$.

\subsection*{Segre varieties}
For projective spaces $\mathbb{P}^{n_1},\ldots,\mathbb{P}^{n_k}$,
their \emph{Segre product}, denoted as
$\operatorname{Seg}(\mathbb{P}^{n_1} \times \cdots \times \mathbb{P}^{n_k})$,
is the image of the Segre map:
\[
\operatorname{Seg}: \, \quad
%\mathbb{P}^{n_1} \times \cdots \times \mathbb{P}^{n_k}
%\to \mathbb{P} \bigotimes_{j=1}^k \mathbb{C}^{n_j+1},
([u_1], \ldots, [u_k]) \mapsto
[ u_1 \otimes \cdots \otimes u_k ].
\]
The dimension of $\operatorname{Seg}(\mathbb{P}^{n_1}
\times \cdots \times \mathbb{P}^{n_k})$
is the sum $n_1+\cdots+n_k$.
%
%We denote by $\operatorname{Seg}(\mathbb{P}^{n_1} \times \cdots \times \mathbb{P}^{n_k})$
%the $\sum_{j=1}^k n_j$-dimensional subvariety of $\mathbb{P}(\bigotimes_{j=1}^k \mathbb{C}^{n_j+1})$, consisting of all decomposable tensors
%$[u_1 \otimes \cdots \otimes u_{k}],u_j\in \mathbb{C}^{n_j+1},j=1,\dots,k$.
%
The Segre product $\mathbb{P}^{n_1},\dots,\mathbb{P}^{n_k}$
is defined by equations of the form
\[
\mA_\alpha \mA_\beta - \mA_\gamma \mA_\gamma=0 \,,\quad
(\alpha + \beta  = \gamma + \tau,\,
\alpha,\beta,\gamma,\tau\in {\prod}_{j=1}^k \{0,\dots, n_j\}
).
\]
Here, tensors in $\mathbb{P}(\bigotimes_{j=1}^k \mathbb{C}^{n_j+1})$
are labelled by integral tuples in $\prod_{j=1}^k \{0,\dots, n_j\}$.
%For example, coordinates on $\mathbb{P}(\mathbb{C}^2 \otimes \mathbb{C}^2)$
%are $x_{0,0},x_{1,0},x_{0,1},x_{1,1}$ and
%$\operatorname{Seg}(\mathbb{P}^1 \times \mathbb{P}^1)\subset \mathbb{P}^3$
%is defined by
%\[
%x_{0,0} x_{1,1} - x_{1,0} x_{0,1} = 0.
%\]

%%%%%%%%%%%%%%%%%%%%%%%%%%%%%%%%%%%%%
\iffalse

\subsection{Grassmannian variety}
We denote by $\Gr(k,n)$ the Grassmannian variety consisting of all
$k$-dimensional linear subspaces of $\mathbb{C}^n$.
We have a Pl\"{u}cker embedding
\[
\Gr(k,n) \to \mathbb{P} \bigwedge^k \mathbb{C}^n,\quad W\mapsto [w_1\wedge \cdots w_k],
\]
where $\{w_1,\dots, w_k\}$ is a basis of $W$
and the dimension of $\Gr(k,n)$ is $k(n-k)$. In particular, we have
$\Gr(1,n) = \mathbb{P}^{n-1}$ and $\Gr(k,n) \simeq \Gr(n-k,n)$.

\fi
%%%%%%%%%%%%%%%%%%%%%%%%%%%%%%%%%%%%

\subsection*{Secant varieties}
\label{ssc:secvar}

Let $X \subseteq \cpx^{n+1}$ be an affine variety
and let $v_d(X)$ be its image under the $d$-th Veronese map $v_d$.
Define the set
%\[
%\sigma^{\circ}_r(v_d(X)) \coloneqq \bigcup_{u_1,\dots, u_r \in X}   \mathbb{P}^{r-1}_{v_d(x_1),\dots, v_d(x_{r}) }
%\]
\[
\sigma^{\circ}_r(v_d(X)) \, := \,
\big\{ (u_1)^{\otimes d} + \cdots + (u_r)^{\otimes d}:
u_1,\dots, u_r \in X  \big \}.
\]
%
%where $x_1,\dots, x_r$ are distinct points in $X$ and
%$\mathbb{P}^{r-1}_{v_d(x_1),\dots, v_d(x_{r})}$
%is the plane of dimension $(r-1)$ spanned by $v_d(x_1),\dots, v_d(x_{r})$.
%
The Zariski closure of $\sigma_r^\circ (v_d(X))$,
which we denote as $\sigma_r(v_d(X)) \subseteq \mathbb{C}^{n+1}$,
is called the \emph{$r$th secant variety} of $v_d(X)$.
The closure $\sigma_r(v_d(X))$ is an affine variety in
$\operatorname{S}^d(\mathbb{C}^{n+1})$,
while $\sigma_r^\circ (v_d(X))$ is usually not.
However, it holds that
\[
\dim \sigma_r^\circ (v_d(X)) = \dim \sigma_r (v_d(X)),
\]
because $\sigma^\circ_r (v_d(X))$ is a dense subset of
$\sigma_r(v_d(X))$ in the Zariski topology.
When $v_d(X)$ is replaced by a general variety $Y$,
the sets $\sigma_r^{\circ}(Y)$ and $\sigma_r(Y)$
can be defined in the same way.
We refer to \cite{Land12} for secant varieties.

\section{Properties of STDX}
\label{Sec:symmetric X-rank}

Let $X \subseteq \cpx^{n+1}$ be a set that is given by homogeneous polynomial equations.
For a given tensor $\mA \in \S^d(\mathbb{C}^{n+1})$,
a \emph{symmetric $X$-decomposition on $X$} is
\begin{equation}
\label{eqn:symmetric $X$-rank decomposition}
\mA = (u_1)^{\otimes d} + \cdots + (u_r)^{\otimes d}, \quad
u_1, \ldots, u_r \in X.
\end{equation}
The smallest such $r$ is called the \emph{symmetric $X$-rank} of $\mA$,
which we denote as $\srank_X(\mA) $, or equivalently,
\be \label{def:srankX(A)}
\srank_X(\mA) = \min \{ r: \,
\mA = \sum_{i=1}^r (u_i)^{\otimes d}, \, u_i \in \cpx^X \}.
\ee
When $r$ is the smallest,
\reff{eqn:symmetric $X$-rank decomposition}
is called a \emph{rank-retaining symmetric $X$-decomposition} for $\mA$.
It is possible that the decomposition
\reff{eqn:symmetric $X$-rank decomposition} does not exist;
for such a case, we define $\srank_X(\mA) = + \infty$.
For instance, a symmetric tensor $\mA$ admits a Vandermonde decomposition
if and only if $\mA$ is a Hankel tensor. So, if $\mA$ is not Hankel,
then $\srank_X(\mA) = + \infty$.
Interested readers are referred to \cite{nie2017hankel,qi2015hankel} for more details.
We denote by $\operatorname{S}^d(X)$ the subspace of tensors
which admit a symmetric $X$-decomposition as in
\reff{eqn:symmetric $X$-rank decomposition}.
As a counterpart for symmetric border rank,
the \emph{symmetric border $X$-rank} of $\mA$ is similarly defined as
\be \label{def:sbrankX(A)}
\sbrank_X(\mA) \, \coloneqq \, \min \{r: \mA \in \sigma_r(v_d(X)) \},
\ee
where $\sigma_r(v_d(X))$ is the secant variety of $v_d (X)$,
defined in Subsection~\ref{ssc:secvar}.
When $\PP X$ is an irreducible variety,
$\sbrank_X(\mA)$ is also equal to the smallest integer $r$
such that $\mA$ is the limit of a sequence of tensors
whose symmetric $X$-rank is $r$
(see \cite[Sec.~5.1.1]{Land12} or \cite[Theorem~2.33]{Mum95}).
The \emph{generic symmetric $X$-rank} of $\S^d(X)$
is the smallest $r$ such that $\sigma_r(v_d(X)) = \S^d(X)$.
When $X = \cpx^{n+1}$, the symmetric $X$-rank becomes
the usual symmetric rank (or Waring rank).
If $\PP X = v_d(\mathbb{P}^1)$ is the rational normal curve,
the symmetric $X$-rank becomes the Vandermonde rank for Hankel tensors \cite{nie2017hankel,qi2015hankel}.
How to characterize tensors that has a
symmetric $X$-decomposition as in
\eqref{eqn:symmetric $X$-rank decomposition}?
How to tell $\srank_X(\mA) < + \infty$
or $\srank_X(\mA) = + \infty$? These questions are the focus of this section.
\subsection{Existence of symmetric $X$-decompositions}
Let $\PP X \subseteq \PP^n$ be the projectivization of $X$. Its vanishing ideal is $\mathcal{I}(X)$, the ideal of polynomials that are identically zero on $X$. The section of degree-$d$ forms in $\mathcal{I}(X)$ is
\[
\mathcal{I}_d(X) \coloneqq \mc{I}(X) \cap \mathbb{C}[x]^h_d.
\]
Let $c := \dim \mathcal{I}_d(X)$. Choose a basis
$
\{f_1, \ldots, f_c\}
$
for $\mc{I}_d(X)$. Let $l_1,\dots, l_c$ be linear functionals
on $\operatorname{S}^d(\mathbb{C}^{n+1})$ such that
\be \label{eqn:pullback}
l_i \Big( {\sum}_j (u_j)^{\otimes d} \Big) =
{\sum}_j f_i(u_j) .
\ee
They are linearly independent functions on $\operatorname{S}^d(X)$.
If we use $\hat{f}_i$ to denote the dehomogenization of $f_i$,
i.e., $\hat{f}_i(x_1,\ldots, x_n) := f_i(1,x_1,\ldots,x_n)$, then
$
l_i (\mA) = \langle \hat{f}_i, \mA \rangle
$
for all $\mA \in \operatorname{S}^d(\mathbb{C}^{n+1})$.
See \reff{<p,mA>} for the definition of the operation $\langle \cdot, \cdot \rangle$.
\begin{proposition}\label{prop:codim}
Let $X,c,f_i, l_i$ be as above. Then, a tensor
$\mA \in \operatorname{S}^d(\mathbb{C}^{n+1})$
belongs to $\operatorname{S}^d(X)$ if and only if
$l_1(\mA) = \cdots = l_c(\mA) = 0$.
Consequently, the codimension of $\operatorname{S}^d(X)$ is $c$, i.e.,
$
\dim \operatorname{S}^d(X) = \binom{n+d}{d} - c.
$
\end{proposition}
\begin{proof}
Clearly, if $\mA \in \operatorname{S}^d(X)$,
then $l_1(\mA) = \cdots = l_c(\mA) = 0$.
Next, we prove the converse is also true.
Suppose $\mA \in \operatorname{S}^d(\mathbb{C}^{n+1})$
is such that $l_1(\mA) = \cdots = l_c(\mA) = 0$.
To show $\mA \in \operatorname{S}^d(X)$,
it is enough to show that:
if $l$ is a linear function vanishing on
$\operatorname{S}^d(X)$ then $l(\mA)=0$.
Each $l_i$ vanishes on $v_d(X)$ and $l_1,\dots, l_c$
are linearly independent as vectors in
$\operatorname{S}^d(\mathbb{C}^{n+1})^\ast$.
(The superscript $^\ast$ denotes the dual space.)
Note that $l \in \operatorname{S}^d(\mathbb{C}^{n+1})^\ast$
and it vanishes on $v_d(X)$.
So, there is a form $f \in \cpx[x]_d$ such that
$
l( u^{\otimes d} ) = f(u)
$
for all  $u \in \cpx^{n+1}$.
Since $l \equiv 0$ on $v_d(X)$, $f$ also vanishes on $X$.
So, $f$ is a linear combination of $f_1,\dots, f_c$,
and hence $l$ is a linear combination of $l_1,\dots, l_c$.
This implies that $l(\mA) = 0$ and $\mA \in \operatorname{S}^d(X)$.

Since $l_1,\dots, l_c$ are linearly independent, the subspace
$\operatorname{S}^d(X)$ are defined by $c$
linearly independent linear equations.
So its codimension is $c$. Since the dimension of
$\operatorname{S}^d(\mathbb{C}^{n+1})$ is $\binom{n+d}{d}$,
the dimension of $\operatorname{S}^d(X)$
follows from the codimension.
\end{proof}
The first part of Proposition~\ref{prop:codim}
is a high dimensional analogue of the apolarity lemma,
which can be found in \cite[Theorem~5.3]{IarKan99}, \cite[Section~1.3]{RS2000},
and \cite[Section~3]{Teitler2015}). For convenience of referencing,
we state this result here and give a straightforward proof.
We would like to thank Zach Teitler for pointing out the relationship between
Proposition~\ref{prop:codim} and the apolarity Lemma.

By Proposition~\ref{prop:codim}, we get the following algorithm
for checking if $\mA \in \operatorname{S}^d(X)$ or not.
Suppose the vanishing ideal $\mathcal{I}(X)$ is generated by the forms
$f_1,\dots, f_s$, with degrees $d_1\le \dots \le d_s$ respectively.

\begin{algorithm}
\label{alg:algorithm1}
For a given tensor $\mA \in \operatorname{S}^d(\mathbb{C}^{n+1})$,
do the following:
\bit

\item [Step~1:] Find the integer $k \geq 0$ such that
$d_{k} \le d < d_{k+1}.$\footnote{Here we adopt the convention that
$d_0 = 0$ and $d_{s+1} = \infty$.}

\item [Step~2:] For each $t = 1,\dots, k$ and
$\beta\in \mathbb{N}^{n+1}_{d- d_{t}}$, let
$
f_{t,\beta} := f_t \cdot x_0^{\beta_0} \cdots x_n^{\beta_n}.
$

\item [Step~3:] Check whether or not $\langle f_{t,\beta}, \mA \rangle = 0$
for all $t$ and $\beta$ in Step~2. If it is,
then $\mA \in \operatorname{S}^d(X)$;
otherwise, $\mA  \not\in \operatorname{S}^d(X)$.

\eit

\end{algorithm}

The above algorithm can be easily applied to detect tensors in
$\operatorname{S}^d(X)$. For instance,
if $X$ is defined by linear equations
$\sum_{j=0}^n f_{ij} x_j = 0$ ($i=1,\ldots,c$), then
$\mA \in \operatorname{S}^d(X)$ if and only if for all
$1 \leq i \leq c$ and $0 \leq k_2, \ldots, k_m \leq n$,
\[
\sum_{j=0}^n f_{ij} \mA_{jk_2\ldots k_m} = 0 .
\]
If $X$ is a hypersurface defined by
a single homogeneous polynomial $f(x) = 0$ with $\deg(f) \leq m$, then
$\mA \in \operatorname{S}^d(X)$ if and only if
\[
\langle f \cdot x^\af,   \mA \rangle = 0
\]
for all monomials $x^\af$ with $\deg(f)+|\af| = m$.

When does $\operatorname{S}^d(X) = \operatorname{S}^d(\mathbb{C}^{n+1})$,
i.e., when does every tensor admit a symmetric $X$-decomposition?
By Proposition~\ref{prop:codim},
we get the following characterization.

\begin{corollary}
\label{corollary:fill ambient space}
Let $X \subseteq \cpx^{n+1}$ be as above.
Then, the equality $\operatorname{S}^d(X) =
\operatorname{S}^d(\mathbb{C}^{n+1})$ holds if and only if
$\mathcal{I}_d(X) = \{0\}$, which is equivalent to that
$X$ is not contained in any hypersurface of degree $d$.
\end{corollary}

The above corollary implies that if $X$ is a hypersurface of degree
bigger than $d$, then every tensor in $\operatorname{S}^d(\mathbb{C}^{n+1})$
has a symmetric $X$-decomposition. Moreover, if $X = \mathbb{C}^{n+1}$,
then obviously we have $\mathcal{I}_d(X) = \{0\}$ for any $d\ge 1$,
which implies the well known fact \cite{CGLM08} that every symmetric tensor admits a symmetric decomposition.

\subsection{The dimension and expected rank}

By Proposition \ref{prop:codim},
$
\dim \operatorname{S}^d(X) = h_{\PP X}(d),
$
where $h_{\PP X}(\cdot)$ is the Hilbert function
for the projective variety $\PP X$.
For the Veronese map $v_d$, we have $\dim v_d (X) = \dim X$.
Therefore, the expected dimension of the secant variety
$\sigma_r (v_d(X))$:
\[
\expdim \sigma_r (v_d(X)) = \min \left\lbrace r \dim X,
h_{\PP X}(d) \right\rbrace.
\]
The \emph{expected generic symmetric $X$-rank}
of $\operatorname{S}^d(X)$ is therefore
\begin{equation}\label{eqn: exp.grank}
\operatorname{exp.grank} =
\left\lceil {h_{\PP X}(d)}/{\dim X} \right\rceil.
\end{equation}

\begin{example}\label{example:segre and hypersurface}
(i) If $\PP X =\operatorname{Seg}(\PP^{n_1} \times \cdots \times \PP^{n_k})$,
the Segre variety, then $\dim X = n_1+\cdots +n_k+1$.
Its Hilbert function
$h_{\PP X}(d)  = \prod_{j=1}^k \binom{n_j + d}{n_j}$
\cite[Example 18.15]{Harris1992}. So
\[
\dim \operatorname{S}^d(X) = \prod_{j=1}^k \binom{n_j + d}{n_j},\quad
\operatorname{exp.grank} = \left\lceil
\frac{\prod_{j=1}^k \binom{n_j+d}{n_j}}{\sum_{j=1}^k n_j +1}
\right \rceil.
\]
(ii) If $\PP X \subseteq \mathbb{P}^n$ is a hypersurface defined
by a form of degree $t$, the Hilbert function of $\PP X$ is
\[
h_{\PP X}(d) =\begin{cases}
\binom{n+d}{n},~\text{if}~d<t, \\
\binom{n+d}{n} - \binom{n+d-t}{n},~\text{otherwise}.
\end{cases}
\]
Then, $\dim \operatorname{S}^d(X) = h_{\PP{X}}(d)$ and
the $\operatorname{exp.grank}$ can be obtained accordingly.
%\[
%\operatorname{exp.grank} = \begin{cases}
%\lceil \frac{\binom{n+d}{n}}{n} \rceil,~\text{if}~d <t,\\
%\lceil \frac{\binom{n+d}{n} - \binom{n+d-t}{n}}{n}\rceil,~\text{otherwise}.
%\end{cases}
%\]
\end{example}

When $\PP X$ is a curve,
we can get the dimension of $\sigma_r(v_d(X))$ as follows.

\begin{proposition}\label{prop:curve}
If $\PP X$ is a non-degenerate curve
(i.e., $\dim \PP X=1$ and $\PP X$ is not contained in any
proper linear subspace of $\PP^{n}$), then
\[
\dim \sigma_r (v_d(X)) = \min
\left\lbrace 2r, h_{\PP X}(d) \right\rbrace.
\]
Therefore, the generic
symmetric $X$-rank is $\lceil \frac{h_X(d) }{2} \rceil$.
Moreover, if $\PP X$ is nonsingular of genus $g$ and degree $t$,
then there exists an integer $d_0$ such that for all $d\ge d_0$,
\begin{align*}
\dim \sigma_r(v_d(X)) =\min\{2r, dt - g+1 \},\quad
\operatorname{grank}  = \lceil \frac{dt - g +1 }{2} \rceil.
\end{align*}
\end{proposition}
\begin{proof}
The first part follows directly from \cite[Example 11.30]{Harris1992}.
The ``moreover" part follows from
Riemann-Roch theorem~\cite[Chapter 4]{Hartshorne1977}.
\end{proof}

By \cite[Example 13.7]{Harris1992}, if
$\PP X$ is a space curve of degree $e$, i.e., $\PP X$ is a curve in $\mathbb{P}^3$ which intersects a generic plane in $e$ points,
then $d_0 = e-2$ in Proposition~\ref{prop:curve}. For an arbitrary curve $\PP X$, however, not much is known about $d_0$.
The Hilbert functions are also known for Veronese varieties,
Grassmann varieties and flag varieties, see \cite{Harris1992,GN2011}. Hence the expected value of the generic rank for them may be calculated by \eqref{eqn: exp.grank}.

%%%%%%%%%%%%%%%%%%%%%%%%%%%%%%%%%%%%%%%%
\iffalse

\subsection{the symmetric $X$-rank decomposition problem}
\label{subsection:SXRDP} Given a symmetric tensor
$\mA \in \S^d(\mathbb{C}^{n+1})$ and a projective variety
$X\subseteq \mathbb{P}^n$, we may consider the symmetric
$X$-rank decomposition problem (SXRDP):
find the smallest positive integer $r$ and $l_1,\dots, l_r\in \hat{X}$ such that
\[
\mA  = l_1^{\otimes d}  + \cdots + l_r^{\otimes d}.
\]
The SXRDP may also be regarded as an optimization problem:
\begin{align*}
&\min \quad \lVert \mA - \sum_{j=1}^r l_j^{\otimes d} \rVert^2\\
& \operatorname{s.t.}\quad l_1,\dots, l_r\in \hat{X}.
\end{align*}
If $X = \mathbb{P}^n$ and $r=1$, then the SXRDP
is the symmetric rank one tensor approximation problem considerted in
\cite{nie2014semidefinite} and for $r\ge 2$,
then the SXRDP is the low symmetric rank tensor
approximation problem considered in \cite{nie2017generating}.

\fi
%%%%%%%%%%%%%%%%%%%%%%%%%%%%%%%%%%

\subsection{Vandermonde decompositions for nonsymmetric tensors}
\label{ssc:VDnst}

A nonsymmetric tensor $\mA \in (\mathbb{C}^{d+1})^{\otimes k}$
is said to admit a \emph{Vandermonde decomposition}
if there exist vectors ($s=1,\ldots, k$, $j=1,\ldots, r$)
\[
v_s^{(j)} := \Big(a_{sj})^{d}, (a_{sj})^{d-1}b_{sj},
\ldots, a_{sj} (b_{sj})^{d-1}, (b_{sj})^{d} \Big)
\in \mathbb{C}^{d+1}
\]
such that
\[
\mA = {\sum}_{j=1}^r v_1^{(j)} \otimes \cdots \otimes v_k^{(j)}.
\]
The smallest integer $r$ in the above
is called the \emph{Vandermonde rank} of $\mA$,
which we denote as $\vrank(\mA)$. Since
$
v_s^{(j)} = (a_{sj},b_{sj})^{\otimes d}\in \S^{d}
\mathbb{C}^2 \simeq \mathbb{C}^{d+1},
$
we can rewrite \eqref{eqn:Vanderdecomp} equivalently as
\be\label{eqn:Vanderdecomp}
\mA
= {\sum}_{j=1}^r \Big((a_{1j},b_{1j})
\otimes \cdots \otimes (a_{kj},b_{kj}) \Big)^{\otimes d}.
\ee
The Vandermonde decomposition can be thought of as a symmetric
tensor decomposition on the set $X\subseteq \mathbb{C}^{2^k}$ such that
\[
\PP X \coloneqq  \mathbb{P}^1 \times \cdots \times \mathbb{P}^1
\quad \mbox{\, ($k$ times)}.
\]
The variety $\sigma_r(v_d(X) ) \subseteq \mathbb{P} \S^d(\mathbb{C}^{2^k})$
is exactly the Zariski closure of tensors
whose Vandermonde ranks at most $r$.
The vanishing ideal of the Segre variety
$\mathbb{P}^{n_1} \times \cdots \times \mathbb{P}^{n_k}$
is generated by $2\times 2$ minors of its flattenings
\cite{grone1977decomposable}. So
$\mathcal{I}_d(\mathbb{P}^{n_1} \times \cdots \times \mathbb{P}^{n_k})
\ne 0$ for all $d\ge 2$. By Corollary~\ref{corollary:fill ambient space},
$\operatorname{S}^d(X)$ is a proper subspace of $\S^d(\mathbb{C}^{2^k})$.
However, every $\mA \in (\mathbb{C}^{d+1})^{\otimes k}$
has a Vandermonde decomposition.

\begin{theorem}
\label{thm:HDnst}
Every tensor in $(\mathbb{C}^{d+1})^{\otimes k}$
has a Vandermonde decomposition.
\end{theorem}
\begin{proof}
Each $ \mA \in (\mathbb{C}^{d+1})^{\otimes k}$ admits
a general tensor decomposition, say,
\[
\mA = {\sum}_{j=1}^L u_{j,1} \otimes \cdots \otimes u_{j,k}
\]
for vectors $u_{j,i} \in \cpx^{d+1}$.
Choose distinct numbers $t_0, t_1, \ldots, t_d$ and let
\[
v_l := (1, t_l, t_l^2, \ldots, t_l^d),
\]
for $l=0,1,\ldots,d$. Clearly,
$v_0, v_1, \ldots, v_d$ are linearly independent
and they span $\cpx^{d+1}$. For each $u_{j,i}$,
there exists numbers $\lmd_{j,i,0}, \ldots, \lmd_{j,i,d}$ such that
\[
u_{j,i} = \lmd_{j,i,0} v_0 + \cdots + \lmd_{j,i,d} v_l.
\]
Plugging the above expression of $u_{j,i}$
in the decomposition for $\mA$, we get
\[
\mA = {\sum}_{j_1, \ldots, j_k=0}^l
c_{j_1, \ldots, j_k} v_{j_1} \otimes \cdots \otimes v_{j_k},
\]
for some scalars $c_{j_1, \ldots, j_k}$.
\end{proof}

\section{Generating polynomials}
\label{sec:GPM}

Assume that $X \subseteq \cpx^{n+1}$
is a set defined by homogeneous polynomial equations.
For a given tensor $\mA\in \S^d(\mathbb{C}^{n+1})$
with $\srank_X(\mA) = r$,
we discuss how to compute the symmetric $X$-decomposition
\be \label{eq:F=sum:ui:tpd}
\mA = (u_1)^{\otimes d} + \cdots + (u_r)^{\otimes d}, \quad
u_1,\dots, u_r \in X.
\ee
Suppose $X$ is given as
\be \label{def:X}
X =\{ x \in \cpx^{n+1}: \,
h_i(x) =  0, i=1,\ldots, N
\},
\ee
with homogeneous polynomials $h_i$ in $x:= (x_0, x_1, \ldots, x_n)$.
Denote $y := (y_1, \ldots, y_n)$. The dehomogenization of $X$ is the set
\be \label{set:Y}
Y =\{ y \in \cpx^n: \,
h_i(1, y_1, \ldots, y_n) =  0, i=1,\ldots, N
\}.
\ee
%
%Let $f(y)$ be the de-homogenization of the form
%$\mA(x_0, x_1, \ldots, x_n)$,
%which is defined as in \reff{df:F(x)}, i.e.,
%\[
%f(y) = \mA(1, y_1, \ldots, y_n).
%\]
%
If each $(u_j)_0 \ne 0$, then the decomposition~\reff{eq:F=sum:ui:tpd}
is equivalent to that
\be \label{eqn:Y-decomp}
\mA = \lmd_1 (1, v_1)^{\otimes d} + \cdots + \lmd_r(1, v_r)^{\otimes d},
\ee
for scalars $\lambda_j\in \mathbb{C}$ and vectors $v_j \in Y$.
In fact, they are
\[
\lambda_j = ( (u_j)_0 )^d, \quad
v_j = \frac{1}{ (u_j)_0 } \big((u_j)_1,\dots, (u_j)_n \big)^\mathsf{T}.
\]
The assumption that $(u_j)_0 \ne 0$ is generic.
For the rare case that it is zero,
we can apply a generic coordinate change for $\mA$
so that this assumption holds.
Throughout this section, we assume the rank $r$ is given.

\subsection{Generating polynomials}

Denote the quotient ring $C[Y]:=\mathbb{C}[y]/ \I(Y)$,
where $I(Y)$ is the vanishing ideal of $Y$.
The $C[Y]$ is also called the coordinate ring of $Y$.
We list monomials in $\mathbb{C}[y]$ with respect to the graded lexicographic order, i.e.,
%
%\footnote{We can use any other monomial ordering here.},
%%
%%% we should delte this footnote, because later we need to use this ordering
%%% othersise, it may cause confusion
%
\[ 1 < y_1 <\cdots < y_n < y_1^2 < y_1y_2 < \cdots .\]
We choose $B_0$ to be the set of first $r$ monomials,
whose images in $\mathbb{C}[Y]$ are linearly independent, say,
\[
B_0 = \{ y^{\beta_1},\dots, y^{\beta_r}\}.
\]
The border set of $B_0$ is
$
B_1  := B_0 \cup y_1B_0 \cup \cdots \cup y_n B_0 .
$
The \emph{boundary} of $B_1$ is
\[
\partial B_1  :=
\big(  B_0 \cup y_1B_0 \cup \cdots \cup y_n B_0 \big) \setminus B_0.
\]
For a matrix
$G \in \cpx^{ r \times \lvert \partial B_1 \rvert }$, we label it as
\[
G = (c_{i,\alpha})_{i\in [r], \alpha\in \partial B_1}.
\]
For each $\alpha\in \partial B_1$, denote the polynomial in $y$
\be \label{phiG:af}
\varphi[G,\alpha] :=
{\sum}_{i=1}^r c_{i,\alpha} y^{\beta_i} - y^{\alpha} .
\ee
Denote the tuple of all such polynomials as
\be \label{set:phiG}
\varphi[G] \, := \, \big( \varphi[G,\alpha](y): \alpha\in \partial B_1 \big).
\ee
Let $J_G$ be the ideal generated by $\varphi[G]$.
The set $\varphi[G]$ is called a
\emph{border basis} of $J_G$ with respect to $B_0$.
We refer to \cite{Stetter2004,Kehrein2005} for border sets and border bases.

In the following, we give the definition of generating polynomials
which were introduced in \cite{nie2017generating}.
For $\alpha\in \partial B_1$, the polynomial $\varphi[G,\alpha]$
is called a \emph{generating polynomial} for $\mA\in \S^d(\mathbb{C}^{n+1})$ if
\be \label{eqn:generatingmatrix}
\langle y^{\gamma} \varphi[G,\alpha] , \mA \rangle = 0 \quad
\forall \, \gamma\in \mathbb{N}^n_{d-\lvert \alpha \rvert}.
\ee
(See \reff{<p,mA>} for the operation $\langle, \rangle$.)
If $\varphi[G,\alpha]$ is a generating polynomial
for all $\alpha\in \partial B_1$,
then $G$ is called a generating matrix for $\mA$.
The set of all generating matrices for $\mA$ is denoted as
$\mathcal{G}(\mA)$. The condition~\eqref{eqn:generatingmatrix}
is equivalent to that
\be \label{eqn:generatingmatrix1}
{\sum}_{i=1}^r   c_{i,\alpha} \mA_{\beta_i+\gamma} =
\mA_{\alpha + \gamma}.
\ee
We use $G(:,\alpha)$ to denote the $\alpha$th column of $G$.
Then \reff{eqn:generatingmatrix1} can be rewritten as
\[
A[\mA,\alpha] G(:,\alpha) = b[\mA,\alpha]
\]
where $A[\mA,\alpha], b[\mA,\alpha]$ are given as
\[
\big( A[\mA,\alpha] )_{\gamma,\beta} = \mA_{\beta + \gamma},\quad
\big( b[\mA,\alpha] )_\gamma = \mA_{\alpha + \gamma} \quad
\big( \beta\in B_0,\,
\gamma \in \mathbb{N}^n_{d - \lvert \alpha \rvert} \big).
\]
The solutions to \eqref{eqn:generatingmatrix1}
can be parameterized as $c_\alpha + N_\alpha w_\alpha$,
where $c_\alpha$ is a solution to \eqref{eqn:generatingmatrix1},
$N_\alpha$ is a basis matrix for the nullspace,
and $w_\alpha$ is the vector of free parameters.
So, every generating matrix can be parameterized as
\be \label{eqn:generatingmatrix2}
G(w) = C + N(w),
\ee
where $C$ is a constant matrix and
$N(w) \in \cpx^{[r] \times \partial B_1 }$.
The following is an example of parameterizing $G(w)$.

\begin{example}\label{ex:smallexample1}
Consider the tensor $\mA \in \S^3(\mathbb{C}^4)$ such that
\begin{multline*}
\mA(1,y_1,y_2,y_3) \, = \,
32\,{y_{{3}}}^{3}-12\,{y_{{3}}}^{2}y_{{2}}+126\,{y_{{3}}}^{2}y_{{1}}-
48\,y_{{3}}{y_{{2}}}^{2}+36\,y_{{3}}y_{{2}}y_{{1}} \\
+150\,y_{{3}}{y_{{1} }}^{2}
-20\,{y_{{2}}}^{3}-18\,{y_{{2}}}^{2}y_{{1}} +42\,y_{{2}}{y_{{1}}}
^{2}+51\,{y_{{1}}}^{3}+18\,{y_{{3}}}^{2}-36\,y_{{3}}y_{{2}} \\
+84\,y_{{3}} y_{{1}}-30\,{y_{{2}}}^{2}+12\,y_{{2}}y_{{1}}
+45\,{y_{{1}}}^{2}+6\,y_{{3}}-6\,y_{{2}}+9\,y_{{1}}-1.
\end{multline*}
For $r=3$, if we choose
$B_0 = \lbrace 1, y_1,y_2  \rbrace$,
%\quad \beta_1 = (0,0,0),\quad \beta_2 = (1,0,0),\quad \beta_3 = (0,1,0).
then
\[
\partial B_1 = \lbrace y_3,y_1^2,y_1y_2,y_2^2,y_1y_3,y_2y_3 \rbrace,
\quad \lvert \partial B_1 \rvert = 6,
\]
\begin{align*}
\alpha_1 &= (0,0,1),\quad \alpha_2 =(2,0,0),\quad \alpha_3 = (1,1,0), \\
\alpha_4 &= (0,2,0),\quad \alpha_5=(1,0,1),\quad \alpha_6 = (0,1,1).
\end{align*}
The $A[\mA,\alpha_i]$ and $b[\mA,\alpha_i]$
can be formulated accordingly. For instance,
\[
A[\mA,\alpha_1] =
\arraycolsep=2.5pt
\kbordermatrix{
& (0,0,0) & (1,0,0) & (0,1,0)  \\
(0,0,0) & -1 & 3 & -2 \\
(1,0,0) & 3 & 15  & 2   \\
(0,1,0) & -2 & 2  & -10  \\
(0,0,1) & 2 & 14  & -6  \\
(2,0,0) & 15  & 51  & 14  \\
(1,1,0) & 2  & 14  & -6  \\
(0,2,0) & -10  & -6  & -20  \\
(1,0,1) & 14  & 50  & 6  \\
(0,1,1) & -6  & 6  & -16  \\
(0,0,2) & 6  & 42  & -4  \\
},\quad b[\mA,\alpha_1] = \arraycolsep=2.5pt
\kbordermatrix{
& ~  \\
(0,0,0) & 2  \\
(1,0,0) &   14  \\
(0,1,0) & -6 \\
(0,0,1) &6  \\
(2,0,0) & 50    \\
(1,1,0) & 6    \\
(0,2,0) & -16    \\
(1,0,1) & 42   \\
(0,1,1) & -4   \\
(0,0,2) & 32   \\
}.
\]
The generating matrix is uniquely determined, i.e.,
\[
G(w) = \left[ \begin {array}{cccccc} -1&-3&-1&{\frac{83}{20}}&-4&{\frac{63}{
20}}\\ \noalign{\medskip}1&4&1&-{\frac{27}{20}}&4&-{\frac{7}{20}}
\\ \noalign{\medskip}1&0&1&{\frac{9}{10}}&1&{\frac{9}{10}}\end {array}
 \right] .
\]
For $r=4$, if we choose
$
B_0 = \lbrace 1, y_1 ,y_2, y_3 \rbrace,
$
then
$
\partial B_1 =\lbrace y_1^2, y_1y_2, y_2^2, y_1y_3, y_2y_3,y_3^2\rbrace.
$
The generating matrix has $6$ parameters, i.e.,
\[
G(w) =\begin{bmatrix}
-3 + w_1 & -1 + w_2 & 83/20 + w_3 &  -4 + w_4 &  63/20 + w_5 & 3/20 + w_6 \\
4 - w_1 & 1 - w_2 & -27/20 - w_3 &  4 - w_4 &  -7/20 - w_5 & 53/20 - w_6 \\
 - w_1 & 1 - w_2 & 9/10 - w_3 &  1 - w_4 &  9/10 - w_5 & 9/10 - w_6 \\
-3 + w_1 &  w_2 & w_3 &   w_4 &   w_5 & w_6 \\
\end{bmatrix}.
\]

%%%%%%%%%%%%%%%%
\iffalse

and $N(w)$ is the $3\times 6$ matrix whose first column is zero, and the $i$-th column is given by
\[
\begin{bmatrix}
3w_{i1} + 2w_{i2} &-w_{i1} & -w_{i2}
\end{bmatrix}^\mathsf{T},\quad 2\le i \le 6.
\]
Here $w_{ik}$ is a parameter for $2\le i \le 6, 1\le k \le 2$. More precisely, we have
\[
\tiny
N(w) = \begin{bmatrix}
-1 & 9 + 3w_{21} + 2w_{22}  & 4 + 3w_{31} + 2w_{32} &10 + 3w_{41} + 2w_{42} & 1 + 3w_{51} + 2w_{52} & 3 + 3w_{61} + 2w_{62}\\
1 & -w_{21}  & -w_{31}  & -w_{41}  & -w_{51} & -w_{61} \\
1 & -w_{22} &  -w_{32} & -w_{42} & -w_{52} & -w_{62}
\end{bmatrix}.
\]
\fi
%%%%%%%%%%%%%%%%%%%%%%%%%%%%%
\end{example}

\subsection{Polynomial division by $\varphi[G]$}
From now on, suppose the vanishing ideal
\[
\I(Y) = \left\langle g_1,\dots, g_N \right\rangle,
\]
generated by $g_1,\dots, g_N \in \cpx[y]$. For $p \in \cpx[y]$,
denote by $\mbox{NF}(p; G)$ the normal form of $p$
with respect to $\varphi[G]$, i.e., $\mbox{NF}(p; G)$
is the remainder of $p$ divided by $\varphi[G]$,
obtained by the Border Division Algorithm~\cite[Proposition 6.4.11]{KL2005}. Here we use the
%{\bf (What is monominal ordering for the polynomial division here?
%The remainder may not be unique if it is not a Grobner basis.)}
Formally, $\mbox{NF}(p;G)$ is the polynomial such that
$p - \mbox{NF}(p;G) \in J_G$, the ideal generated by polynomials in $\varphi[G]$.
Note that $\mbox{NF}(p; G)$ is a polynomial in $y:=(y_1, \ldots, y_n)$
whose coefficients are parametrized by the entries of $G$.

\begin{proposition}\label{proposition:normalform}
Suppose $B_0$ is the set of first $r$ monomials\footnote{We remark that instead of the first $r$ monomials, one may choose any $r$ monomials which are connected to one. Interested readers are referred to \cite{Mourrain1999} for a detailed discussion. For simplicity, we use in this paper the set of first $r$ monomials, which is obviously connected to one. This choice of basis will be convenient in practical computations like those in Section~\ref{sec:applications}.
%{\bf However, Monique Laurent (a generalized flat extension theorem) pointed out that
%the basis should be connected to one, otherwise the conclusion is not true.
%Should we insist that the basis is connected to one?}
}
$y^{\beta_1},\dots,y^{\beta_r}$ in the graded lexicographic ordering
such that their images in
$\cpx[y]/\I(Y)$ are linearly independent.
Then, a polynomial $p \in \cpx[y]$ belongs to
$J_G$ if and only if $\mbox{NF}(p;G) = 0$.
\end{proposition}
%(\red{**In the above, we do not need $B_0$ to be the set of FIRST
%$r$ monomials in the graded lexicographic ordering?
%Anyone with linearly independent $r$ monomials are okay, right?
%Is this true?}\blue{I think we do need $B_0$ to be the set of first $r$ monomials which are linearly independent, otherwise, the result we cite below can not be applied.})
%(\orange{ ****Although we may not need to choose the first $r$ linearly independent  monomials, the set $B_0$ should be connected to $1$: if
%$m\in B_0$, then either $m=1$, or there exists $i \in [1, n]$ and $m' \in B_0$ such that $m=x_i m'$.  I am not sure whether  the first $r$ linearly independent monomials are always connected to $1$? If not, then we may add this condition.}
%{\bf I got contradictive info.
%In the footnote, we remark that the connectedness to one is not required?
%But Mourrain, Monique Laurent insisted the connectedness to one,
%in their statements of theorems.
%Is there a contradiction?
%From what we wrote, I cannot conclude that
%if the condition of connectedness to one is really necessary or not?}
%)

\begin{proof}
By the construction, $\varphi[G]$ is a border basis of $J_G$,
%({\bf I guess we do not need Proposition~\ref{prop:basis} in the proof.
%By the construction, $\varphi[G]$ is always a border basis of $J_G$,
%isn't this true?})
so it contains a Gr\"{o}bner basis (say, $S$) of $J_G$ with respect to the graded lexicographic ordering.
Indeed, those elements in $G$ which are associated to the corner of $B_0$ form a Gr\"{o}bner basis.
See \cite[Proposition 2.30]{Stetter2004} or \cite[Proposition 4.3.9]{Kehrein2005}
for definition of the corner of $B_0$ and more details of the proof.
%({\bf I checked Stetter's book.
%He did not explicitly conclude that it containts a Grobner basis.
%How can we write down the rigorous proof for Grobner basis?
%I do not follow completely here.})
Therefore, $p\in J_G$ if and only if $\mbox{NF}(p;G) = 0$.
Since $S \subseteq \varphi[G]$, the normal form of $p$
with respect to $S$ is the same as the normal form with respect to
$\varphi[G]$, which is $\mbox{NF}(y;G)$.
Therefore, $p \in J_G$ if and only if
$\mbox{NF}(p;G) = 0$.
\end{proof}

%The normal form $\mbox{NF}(p;G)$ can be obtained
%by a sequence of polynomial divisions. The division of
%$p(y)$ by $q(y)$ is given by
%\[
%\widetilde{p}(y) \coloneqq p(y) - c_\af \frac{y^\alpha}{\operatorname{lm}(q(y))} q(y)
%\]
%where $\operatorname{lm}(q(y))$ is the leading monomial\footnote{
%We assume that $\operatorname{lm}(q(y))$ divides $y^\alpha$.
%Otherwise, $p(y)$ is not reducible by $q(y)$.}
%and $y^\alpha$ is the leading monomial in $p(y)$
%whose coefficient is $c_\alpha$.

The condition $\mbox{NF}(p;G) = 0$ is equivalent to that
its coefficients are zeros identically.
The coefficients of $\mbox{NF}(p;G)$ are polynomials in $c_{i,\af}$.
Their degrees can be bounded as follows.
For a degree $k > 1$, define the set
$B_k$ recursively as
\[
B_k \, := \, B_{k-1} \cup y_1 B_{k-1} \cup \cdots \cup y_n B_{k-1}.
\]
The monomials in $B_k$ generate a subspace,
which we denote as $\mbox{Span}\,B_k$.

\begin{lemma}\label{lemma:degreebound}
Let $B_k, p(y)$ and $\mbox{NF}(p;G)$ be as above.
Write $p=p_1 + p_2$, where
$p_1 \in \cup_{i \geq 1} \mbox{Span} \, B_i$
and $p_2$ is a polynomial whose monomials are not contained in any $B_i$.
If $p_1 \in \mbox{Span} \, B_k$, then the coefficients of
$\mbox{NF}(p;G)$ are polynomials of degree at most $k$ in $G$.
\end{lemma}
\begin{proof}
Note that
$\mbox{NF}(p;G) = \mbox{NF}(p_1;G) + p_2$,
because $p_2$ is not reducible by $\varphi[G]$.
The coefficients of $p_2$ do not depend on $G$.

For the case $k=1$, we can write $p_1 \in \operatorname{Span} B_1$ as
\[
p_1 =  {\sum}_{\beta \in \partial B_1} a_\beta y^\beta +
{\sum}_{\gamma\in B_1} a_\gamma y^\gamma,
\]
with coefficients $a_\beta, a_{\gamma}\in \mathbb{C}$.
The reduction of $p$ by $\varphi[G]$
is equivalent to that
\[
\mbox{NF}(p_1;G) \, = \, p_1 +
{\sum}_{\beta\in \partial B_1} a_\beta \varphi[G,\beta]
= {\sum}_{\gamma\in B_1} a'_\gamma(G) y^\gamma,
\]
where $a'_{\gamma}(G)$ is affine linear in the entries of
$G$ and are affine linear in the coefficients of $p$.
So, the coefficients of
$\mbox{NF}(p;G)$ are affine linear in $G$.

For the case $k > 1$, we can write each $p_1$ as
\[
p_1 = p_0 + {\sum}_{\beta \in \partial B_k} a_\beta y^\beta,
\]
with $p_0 \in \operatorname{Span} B_{k-1}$
and coefficients $a_\beta \in \cpx$.
For each $\bt \in \partial B_k$, denote by
$i(\bt)$ the smallest $i\in [n]$ such that
$\bt \in y_i B_{k-1}$. Then
\[
p_1 = p_0 +  {\sum}_{\beta \in \partial B_k} a_\beta
y_{i(\bt)} y^{ \bt - e_{i(\bt)}  },
\]
\[
\mbox{NF}(p_1;G) = \mbox{NF}(p_0;G) +
{\sum}_{\beta \in \partial B_k} a_\beta
\mbox{NF}\Big(y_{i(\bt)} \mbox{NF}(y^{ \bt - e_{i(\bt)}};G); G \Big).
\]
By induction, the coeffieints of
$\mbox{NF}(p_1;G)$ are of degree $\le k$ in $G$.
\end{proof}

\begin{example}\label{ex:smallexample2}
Suppose that $X$ is the Segre variety
$\mathbb{P}^1 \times \mathbb{P}^1\subset \mathbb{P}^3$.
Then $Y$ is the surface in $\mathbb{C}^3$ defined by
$
g(y) \coloneqq y_{3} - y_{1} y_{2}= 0,
$
where $y_{1},y_{2}$ and $y_{3}$ are coordinates\footnote{
If we use the notation in Proposition \ref{lemma:degreebound},
coordinates of $\mathbb{C}^3$ should be $y_{1,0},y_{0,1}$ and $y_{1,1}$.}
of $\mathbb{C}^3$. If we choose
$
B_1 = \lbrace 1, y_1 ,y_2, y_3\rbrace,
$
then
\[
\partial B_1 = \lbrace y_1^2,y_1y_2,y_2^2,y_1y_3,y_2y_3, y_3^2 \rbrace.
\]
As in Example~\ref{ex:smallexample1}, we can get
\begin{align*}
\varphi[G,(2,0,0)] &= (-3 + w_1) +(4 - w_1) y_1  - w_1 y_2 + w_1 y_3   - y_1^2, \\
\varphi[G,(1,1,0)] &= (-1 + w_2) +(1-w_2) y_1 + (1- w_2) y_2 + w_2 y_3 - y_1y_2,\\
\varphi[G,(0,2,0)] &=(83/20 + w_4) + (-27/20 - w_4) y_1 + (9/10-w_4) y_2 + w_4  y_3  - y_2^2,\\
\varphi[G,(1,0,1)] &= (-4 + w_3) + (4 - w_3) y_1 + (1-w_3) y_2 + w_3  y_3 - y_1y_3,\\
\varphi[G,(0,1,1)] &= (63/20 + w_5) + (-7/20 - w_5) y_1 + (9/10-w_5) y_2 + w_5  y_3  - y_2y_3,\\
\varphi[G,(0,0,2)] &= (3/20 + w_6) + (-53/20 - w_6) y_1 + (9/10-w_6) y_2 + w_6  y_3  - y_2y_3.
\end{align*}
The normal form $NF(g;G)$ of $g$ is
\[
NF(g;G)=(1-w_2) - (1-w_2)y_1 - (1-w_2) y_2 +  (1 - w_2) y_3.
\]
The condition $NF(g;G) = 0$ requires that $1-w_2 = 0$.
\end{example}

\subsection{Commutativity conditions}
For each $1\le i \le n$
and $G \in \cpx^{[r] \times \partial B_1}$,
define the matrix $M_i(G) \in \cpx^{r \times r}$ as
\[
(M_i(G))_{s,t} = \begin{cases}
1,\quad~\text{if}~\beta_s = \beta_t + e_i\in B_0, \\
0, \quad ~\text{if}~\beta_s \ne \beta_t + e_i\in B_0, \\
c_{s,\beta_t + e_i},\quad ~\text{if}~\beta_t + e_i \in \partial B_1.
\end{cases}
\]
They are also called multiplication matrices for the ideal $J_G$.

\begin{theorem}
\label{thm:commuting}
Let $B_0$ be the set of first $r$ monomials
whose images in $\mathbb{C}[Y]$ are linearly independent.
Then, for $G \in \cpx^{ [r] \times \partial B_1}$,
the polynomial system
\be \label{eqn:finitesolution}
\varphi[G](y) = 0
\ee
has $r$ solutions (counting multipliciites)
and they belong to $Y$ if and only if
\be \label{eqn:commutativity}
\big[ M_i(G), M_j(G) \big]   =0 \quad (1\le i < j \le n ),
\ee
\be \label{eqn:extracondition}
\mbox{NF}(g_i;G) = 0  \quad (i=1,\dots, N),
\ee
where $\left\langle g_1,\ldots, g_N \right\rangle = I(Y)$.
\end{theorem}
\begin{proof}
By Theorem~2.4 of \cite{nie2017generating},
\eqref{eqn:finitesolution} has $r$ solutions if and only if
\eqref{eqn:commutativity} holds. All the solutions are contained in
$Y$ if and only if $\I(Y) \subseteq J_G$,
which is equivalent to that
each normal form $\mbox{NF}(g_i;G) = 0$,
by Proposition \ref{proposition:normalform}.
\end{proof}

%{\bf
%I am thinking about a relevant but very important question.
%In Theorem~\ref{thm:commuting}, if $G$ satisfies the commutative
%conditions there but $\varphi[G](y) = 0$ has a repeated solution,
%then is the following true or not?
%\begin{quote}
%There exists a sequence of generating matrices $G_k$ such that $G_k \to G$,
%each $G_k$ satisfies the commutative conditions
%but $\varphi[G_k](y) = 0$ has $r$ distinct solutions.
%\end{quote}
%If the above is true, then we can conclude that the tensor
%$\mA$ has border $X$-rank $r$? If the above is not true,
%can we get a counterexample?
%I spent a lot of time on this question, did not succeed.
%}

\section{An algorithm for computing STDXs}
\label{sc:frm}

Let $X,Y$ be the varieties as in \reff{def:X} and \reff{set:Y}.
This section discusses how to compute
a symmetric tensor decomposition on $X$
for a given tensor $\mA \in \S^d(\mathbb{C}^{n+1})$
whose symmetric $X$-rank is $r$.
To compute \reff{eq:F=sum:ui:tpd}, it is enough to compute
\be \label{eqn:affinedecomp}
\mA =  \lambda_1 (1, v_1)^{\otimes d}
+ \cdots + \lambda_r (1, v_r)^{\otimes d}
\ee
for vectors $v_1, \ldots, v_r \in Y$
and scalars $\lmd_1, \ldots, \lmd_r$.

Choose $B_0 = \{ x^{\bt_1}, \ldots, x^{\bt_r} \}$
to be the set of first $r$ monomials that are linearly independent in $C[Y]$.
For a point $v\in \mathbb{C}^n$, denote by
$B_0(v)\in \mathbb{C}^r$ the column vector
whose $i$th entry is $v^{\beta_i}$. Denote the Zariski open subset
$D$ of $(\mathbb{C}^{n})^r$
\[
D = \{
(v_1, \ldots, v_r) \in (\mathbb{C}^{n})^r: \,
\det  \begin{bmatrix}
B_0(v_1) & \cdots B_0(v_r)
\end{bmatrix} \ne 0
\}.
\]
Recall that $\varphi[G]$ is the tuple
of generating polynomials as in \reff{set:phiG} and
$\mathcal{G}(\mA)$ denotes the set of
generating matrices for $\mA$.

\begin{theorem}
\label{thm:validation}
Let $X,Y,D$ be as above.
For each $\mA \in \S^d(\mathbb{C}^{n+1})$, we have:
\bit

\item
If $G\in \mathcal{G}(\mA)$ and $v_1,\dots, v_r\in Y$
are distinct zeros of $\varphi[G]$,
then there exist scalars $\lmd_1, \ldots, \lmd_r$
satisfying \reff{eqn:affinedecomp}.
%
%\be \label{eqn:affinedecomp}
%f(y) = \sum_{i=1}^r \lambda_i (1 + v_i^\mathsf{T} y)^d.
%\ee

\item
If the decomposition \eqref{eqn:affinedecomp} holds for $\mA$
and $(v_1,\dots,v_r)\in D$, then there exists a unique
$G\in \mathcal{G}(\mA)$ such that
$v_1,\dots, v_r$ are zeros of $\varphi[G]$.

\eit
\end{theorem}
\begin{proof}
The conclusions mostly follow from Theorem~3.2 of
\cite{nie2017generating}.
The difference is that we additionally require
the points $v_1,\ldots, v_r \in Y$.
\end{proof}

According to Theorem \ref{thm:validation},
to compute a symmetric $X$-rank decomposition for $\mA$,
we need to find a generating matrix $G \in \mathcal{G}(\mA)$ such that
\begin{enumerate}
\item $\varphi[G]$ has $r$ distinct zeros.
\item The zeros of $\varphi[G]$ are contained in $Y$,
i.e., $\I(Y) \subseteq J_G$.
\end{enumerate}
Conditions $(1)$ and $(2)$ are equivalent to \eqref{eqn:commutativity} and \eqref{eqn:extracondition} by Theorem \ref{thm:validation}.
%
%Assume that $X \subseteq \cpx^{n+1}$ is given by homogeneous
%polynomial equations. Let $Y$ be its dehomogenization.
%
Suppose the vanishing ideal
$I(Y) = \left\langle g_1,\dots, g_N \right\rangle$.
We have the following algorithm.

\begin{algorithm}\label{alg:framework}
For a given $\mA \in \S^d(\mathbb{C}^{n+1})$
with $\srank_X(\mA) \leq r$, do the following:
\bit

\item [Step~0] Choose the set of first $r$ monomials
$y^{\beta_1},\dots, y^{\beta_r}$
that are linearly independent in the quotient ring
$\mathbb{C}[Y] := \mathbb{C}[y]/\I(Y)$,
with respect to the graded lexicographic monomial order.

\item [Step~1] Parameterize the generating matrix
$G(w) =C + N(w)$ as in \eqref{eqn:generatingmatrix2}.

\item [Step~2] For each $g_i$, compute
$\mbox{NF}(g_i;G(w))$ with respect to $\varphi[G(w)]$.

\item [Step~3] Compute a solution of the polynomial system
\be \label{eqn:commuting condition}
\left\{ \begin{array}{rl}
\big[ M_i(G(w)),  M_j(G(w)) \big]   &= 0 \quad (1\le i < j \le n), \\
\mbox{NF}(g_i;G(w)) & = 0 \quad (1\le i \le N).
\end{array} \right.
\ee

\item [Step~4] Compute $r$ zeros $v_1,\dots, v_r\in \mathbb{C}^n$
of the polynomial system
\[
\varphi [G(w),\alpha] = 0 \quad (\alpha\in \partial B_1).
\]

\item [Step~5] Determine scalars $\lambda_1,\dots, \lambda_r$
satisfying \reff{eqn:affinedecomp}.

\eit
\end{algorithm}

The major task of Algorithm~\ref{alg:framework}
is in Step~3. It requires to solve a set of polynomial system,
which is given by commutative equations and normal forms.
The commutative equations are quadratic in the parameter $w$. The equations
$\mbox{NF}(g_i;G) =0$ are polynomial in $w$,
whose degrees are bounded in Lemma~\ref{lemma:degreebound}.
One can apply the existing symbolic or
numerical methods for solving polynomial equations.
In Step~4, the polynomials $\varphi[G(w),\af]$
have special structures. One can get a Gr\"{o}bner basis
quickly, hence their zeros can be computed efficiently.
We refer to \cite{Corless1997A} and
\cite[Sec.~2.4]{nie2017generating}
for how to compute their common zeros.

\begin{remark} \label{rmk:val(r)}
In Algorithm~\ref{alg:framework}, we need to know a value of $r$,
with $r \geq \srank_X(\mA)$. Typically, such a $r$ is not known.
In practice, we can choose $r$ heuristically. For instance,
%when $\mA$ is generically given,
we can choose $r$ to be the expected generic rank given in
the subsection~3.2. If the flattening matrices of $\mA$
have low ranks, we can choose $r$ to be the maximum of their ranks.
For any case, if the system \reff{eqn:commuting condition}
cannot be solved, we can increase the value of $r$
and repeat the algorithm.
\end{remark}

We conclude this section with an example
of applying Algorithm \ref{alg:framework}.

\begin{example}
%\label{ex:smallexample2}
Let $\mA \in \S^3(\mathbb{C}^4)$ be the same tensor as in
Example~\ref{ex:smallexample1}. Let $X \subseteq \cpx^4$
be the set whose dehomogenization $Y\subseteq \mathbb{C}^3$ is the surface whose defining ideal is $I(Y) = \langle y_3 - y_1 y_2 \rangle $.
%
%We notice that $f(y)$ has a symmetric $X$-decomposition of length four:
%\[
%f(y) = (1 + y_1 + y_2 +y_3)^3 + 2(1 + 3y_1 + y_2 + 3 y_3)^3
%- (1 + y_1 - y_2 -y_3)^3 - 3(1 + y_1 + 2y_2 + 2 y_3)^3.
%\]
%This implies $\srank_X(\mA) \le 4$.
%
The maximum rank of flattening matrices of $\mA$ is $3$,
so we apply Algorithm~\ref{alg:framework} with $r=3$.
%(\red{What is the maximum rank of the flattening matrices?
%Is it $3$? We should do the same thing as we have said in the Remark?})
Choose $B_0 = \{1,y_1,y_2\}$. From the calculation in
Example~\ref{ex:smallexample1}, we can get
\begin{align*}
M_{1}(G(w)) &=  \left[ \begin {array}{rrr} 0&-3&-1\\ \noalign{\medskip}1&4&1
\\ \noalign{\medskip}0&0&1\end {array} \right],  \quad
M_{2}(G(w)) =  \left[ \begin {array}{rrr} 0&-1&{\frac{83}{20}}\\ \noalign{\medskip}0
&1&-{\frac{27}{20}}\\ \noalign{\medskip}1&1&{\frac{9}{10}}\end {array}
 \right], \\
M_{3}(G(w)) &=  \left[ \begin {array}{rrr} -1&-4&{\frac{63}{20}}\\ \noalign{\medskip}
1&4&-{\frac{7}{20}}\\ \noalign{\medskip}1&1&{\frac{9}{10}}\end {array}
 \right].
\end{align*}
The matrices $M_1(G(w)),M_2(G(w))$ and $M_3(G(w))$ commute.
The generating polynomials are
\begin{align*}
\varphi[G,(0,0,1)] &= (-1 + y_1 + y_2) - y_3, \\
\varphi[G,(2,0,0)] &= (-3 + 4y_1 ) - y_1^2,\\
\varphi[G,(1,1,0)] &= (-1 + y_1 + y_2) -y_1y_2 ,\\
\varphi[G,(0,2,0)] &= (83/20 -27/20 y_1 + 9/10 y_2) -y_2^2 ,\\
\varphi[G,(1,0,1)] &=  (-4 + 4y_1 + y_2) -y_1y_3  ,\\
\varphi[G,(0,1,1)] &=(63/20 -7/20 y_1 +9/10 y_2 ) -y_2y_3.\\
\end{align*}
They have $3$ common zeros. There are no radical formulae for them,
but they can be numerically evaluated as
\[
(3,1,3),\quad (1,-1.283,-1.283),\quad (1,2.183,2.183).
\]
The given symmetric $X$-decomposition for $\mA$ as
\begin{align*}
\widetilde{\mA}(y) &\coloneqq 2 \, \left(  1+ 3\,y_{{1}}+ \,y_{{2}}+ 3 \,y_{{3}} \right) ^
{3}- 0.7353\, \left(  1+ \,y_{{1}}- 1.283\,y_{{2}}- 1.283\,y_{{3}
} \right) ^{3} \\
&- 2.265\, \left(  1+ \,y_{{1}}+ 2.183\,y_{{2}}+
 2.183\,y_{{3}} \right) ^{3}.
\end{align*}
Because of numerical errors, we do not have $\widetilde{\mA} = \mA$
exactly, but the round-off error
$\lVert \widetilde{\mA} -\mA  \rVert \approx 6.84 \cdot 10^{-16}$.
\end{example}

\section{Numerical experiments}
\label{sec:applications}

In this section, we present examples of applying
Algorithm~\ref{alg:framework} to compute symmetrix $X$-rank decompositions.
The computation is implemented in a laptop with a 2.5 GHz Intel Core i7 processor.
The software for carrying out numerical experiments is {\tt Maple 2017}.
We solve the system~\reff{eqn:commuting condition} by the built-in
function {\tt fsolve} directly.
%(\red{What software is used? Maple?
%What method is used to solve \reff{eqn:commuting condition}?})
%For convenience, we label a tensor $\mA$ as in \reff{newidx:af}.
The algorithm returns a decomposition
\[
\widetilde{\mA} \coloneqq   \widetilde{\lambda}_1 (1, \widetilde{v}_1)^{\otimes d}
+ \cdots + \widetilde{\lambda}_r (1, \widetilde{v}_r)^{\otimes d}.
\]
Because of round-off errors, the equation $\widetilde{\mA} = \mA$
does not hold exactly.
We use the absolute error $\lVert \mA -\widetilde{\mA} \rVert$
or the relative one
$\lVert \mA -\widetilde{\mA} \rVert/\lVert \mA \rVert$
to verify the correctness. Here, the Hilbert-Schmidt norm of $\mA$ is used, i.e.,
\[
\|\mA\|  = \Big(
{\sum}_{i_1,\ldots, i_m }  |\mA_{i_1 \ldots i_m}|^2
\Big)^{1/2}.
\]
We display the computed decompositions by showing
\[
\widetilde{V} = \bbm \widetilde{v}_1 \\ \vdots \\
\widetilde{v}_r \ebm, \quad
\widetilde{\Lambda} = \bbm \widetilde{\lmd}_1 \\ \vdots \\
\widetilde{\lmd}_r \ebm .
\]
For neatness, only four decimal digits are shown,
and $i$ denotes the unit pure imaginary number.
If the real or imaginary part of a complex number is
smaller than $10^{-10}$, we treat it as zero and do not display it,
for cleanness of the paper. To apply Algorithm~\ref{alg:framework}, we need a value
for the rank $r$. This issue is discussed in Remark~\ref{rmk:val(r)}.
%%%%%%%%%%%%%%%%%%%%%%%%%
\iffalse
\red{To decompose the given tensor $\mA$ on the variety $X$ more efficiently, we could first estimate a lower bound for $\rank_X(\mA)$. The simplest such a lower bound is provided by ranks of flattening matrices of $\mA$. For example, if a third order symmetric tensor $\mA \in S^3 \mathbb{C}^n\subseteq \mathbb{C}^n \otimes (\mathbb{C}^n \otimes \mathbb{C}^n)$ can be regarded as an $n \times n^2$ matrix $\mA_{1}$. If $\mA$ is generic, then $\mA_{1}$ has rank $n$ and this gives us a lower bound for $\srank_X(\mA)$. Once we determine a lower bound $r_0$, we may apply Algorithm~\ref{alg:framework} with $r = r _0$. If one of linear systems \eqref{eqn:generatingmatrix1} is not consistent, then it implies that the $\srank_X (\mA) > r_0$ and we may start over with $r =r _0 + 1$. We repeat this process until all systems \eqref{eqn:generatingmatrix1} are consistent.}

\fi
%%%%%%%%%%%%%%%%%%%%%%%%%%%%%%%%%%%%%%
In our computation, we initially choose $r$ to be the maximum
rank of the flattening matrices of the given tensor $\mA$.
If the equations~\eqref{eqn:generatingmatrix1}
or \eqref{eqn:commuting condition} are inconsistent,
we need to increase the value of $r$ by one,
until Algorithm~\ref{alg:framework} successfully returns a decomposition.

For the set $Y$ dehomogenized from $X$ as in \reff{set:Y},
we need generators of its vanishing ideal $I(Y)$.
For some $Y$, it is easy to compute the generators of $I(Y)$;
for some $Y$, it may be difficult to compute them.
This is a classical, standard problem in symbolic computation.
We refer to \cite{ECW1992,FGT2002,GBG1988} for the related work.
So, we do not focus on how to compute generators of $I(Y)$ in this paper.
In our examples, the generators of $I(Y)$ are known or can be
computed easily.

%\red{In all our examples, we need to display generators of
%$I(Y)$, sine this is a crucial part of computation. }

First we illustrate how to apply Algorithm~\ref{alg:algorithm1}
to detect the existence of the STDX for a given $\mA$ and $\PP X$.
\begin{example}
We consider $\mA \in \S^3(\cpx^4)$ that is given as 
\[
\mA_{ijk} = i + j + k
\]
and $X\subseteq \PP^3$ that is defined by $x_2^2 = x_1^2+x_0^2$, $x_3^2 = x_2^2+x_1^2$.
%The vanishing ideal of $Y$ is
%\[
%I(Y) = \langle  y_2^2 - y_1^2 - 1,y_3^2 - y_2^2 - y_1^2 \rangle .
%\]
According to Algorithm~\ref{alg:algorithm1}, we have
\[
f_1 = x_2^2 - x_1^2 - x_0^2,\quad f_2 = x_3^2 - x_2^2 - x_1^2.
\]
We let $\beta \in \mathbb{N}^4_{1}$ be $\beta = (1,0,0,0)$ and hence we have $f_{1,\beta} = f_1 x_0$. It is straightforward to verify that $\langle f_{1,\beta}, \mA \rangle \ne 0$
and hence $\mA$ has no STDX for $X$. Another example is
\[
\mA \in \S^4(\cpx^3), \quad
\mA_{ijkl} = (i + j + k + l)^2 - (i^2+j^2+k^2+l^2)
\]
and $X\in \PP^2$ is defined by $x_0x_1 + x_1 x_2 + x_0 x_2 =0,
x_0^2x_1 + x_1^2 x_2 + x_2^2 x_0 =0$. By Algorithm~\ref{alg:algorithm1} again,
we can show easily that $\mA$ does not admit an STDX for such $X$.
\end{example}

The resting examples in this section are devoted to exhibit the validity and efficiency of Algorithm~\ref{alg:framework}.
To this end, we make the following convention on the representations of tensors.
Recall that a tensor $\mathcal{A} \in \S^d(\mathbb{C}^{n+1})$
is an array of numbers whose elements are indexed by $(i_1,\dots,i_d)$, i.e.,
$\mathcal{A} = (\mathcal{A}_{i_1,\dots, i_d})$, where $0\le i_1,\dots, i_d \le n$.
As in Section~\ref{ssc:equdes}, we can equivalently represent $\mathcal{A}$ by $\mathcal{A}_\alpha$'s,
where $\alpha = (\alpha_0,\dots,\alpha_n)\in \mathbb{N}_d^{n+1}$ satisfies $\lvert \alpha \rvert = d$.
To be more precise, for each element in $\{(i_1,\dots, i_d): 0 \le i_1,\dots, i_d \le n\}$, we have
\[
\mathcal{A}_{i_1,\dots, i_d} = \mathcal{A}_{\alpha},
\]
where $\alpha \in \mathbb{N}_d^{n+1}$ is the sequence such that $\lvert \alpha  \rvert = d$
and $x_0^{\alpha_0} \cdots x_n^{\alpha_n}= x_{i_1} \cdots x_{i_d}$.
We may list the entries
$\mathcal{A}_\alpha$ with respect to the lexicographic order, i.e.,
$\mathcal{A}_\alpha$ precedes $\mathcal{A}_\beta$ if and only if
the most left nonzero entry of $\af-\bt$ is positive.
%there exists an integer
%$-1 \le s \le n-1 $ such that $\alpha_0 = \beta_0, \dots, \alpha_s = \beta_s$
%but$\alpha_{s+1} > \beta_{s+1}$.
For instance, a binary cubic tensor
$\mathcal{A} = \in \S^3(\mathbb{C}^2)$
can be displayed as $\mathcal{A}_{30}, \mathcal{A}_{21}, \mathcal{A}_{12}, \mathcal{A}_{03}$.
In the following examples, we will represent a symmetric tensor $\mA$ in this way.

\begin{example}
Let $\PP X\subseteq \mathbb{P}^2$ be the parabola defined by
$x_2^2 - x_0x_1 + x_0^2 = 0$.
Let $\mA \in \S^3(\mathbb{C}^3)$ be the symmetric tensor such that
\begin{align*}
\mA_{300}& = 15, \mA_{210} = 81,\mA_{201} = -6,\mA_{120} = 621,\mA_{111} = -108, \mA_{102} =66,\\
\mA_{030} &=5541, \mA_{021} = -1296,\mA_{012} = 540,\mA_{003} = -102.
\end{align*}
The vanishing ideal of $Y \subseteq \mathbb{C}^2$ is
\[
I(Y) = \langle  y_2^2 - y_1 + 1  \rangle .
\]
The maximum rank of flattening matrices of $\mA$ is $3$.
When we run Algorithm~\ref{alg:framework} with $r=3$,
it does not give a desired tensor decomposition.
So, we apply Algorithm~\ref{alg:framework} with $r = 4$ and the symmetric $X$-decomposition with
\[
\widetilde{V} =  \left[ \begin {array}{cc}  9.874- 0.002786\,i&- 2.979+ 0.0004677\,i
\\ \noalign{\medskip} 1.151- 0.006293\,i&- 0.3886+ 0.008096\,i
\\ \noalign{\medskip} 4.124+ 0.02027\,i& 1.768+ 0.005734\,i
\\ \noalign{\medskip} 16.62+ 1.251\,i& 3.956+ 0.1581\,i\end {array}
 \right]
,\quad \widetilde{\Lambda} =   \left[ \begin {array}{c}  5.184+ 0.0036
\,i\\ \noalign{\medskip} 3.694+ 0.0184\,i
\\ \noalign{\medskip} 6.098- 0.0154\,i
\\ \noalign{\medskip} 0.0246- 0.0066\,i
\end {array} \right].
\]
We have $\lVert {\mA} \rVert = 7241.79$
and the error $\lVert \mA - \widetilde{\mA} \rVert = 8 \cdot 10^{-14}$.
%\quad \frac{\lVert \mA - \widetilde{\mA} \rVert }
%{\lVert \widetilde{\mA} \rVert} =  10^{-17}.
\end{example}

\begin{example}
Let $\PP X \subseteq \PP^2$ be the nodal curve defined by
$
x_1^3 + x_0 x_1^2 - x_0x_2^2 = 0.
$
The vanishing ideal of $Y \subseteq \mathbb{C}^2$ is
\[
I(Y) = \langle  y_1^3 + y_1^2 - y_2^2 \rangle .
\]
Let $\mA \in \S^3(\mathbb{C}^3)$ be the symmetric tensor such that
\begin{align*}
\mA_{300} &=  3,\mA_{210} =24, \mA_{201} = 72, \mA_{120} =
144,  \mA_{111}  = 456,
\mA_{102} =
1224, \\
\mA_{030}  &=  1080,
 \mA_{021} = 3288,
 \mA_{012} =9432,
  \mA_{003} =
28512.
\end{align*}
The maximum rank of flattening matrices of $\mA$ is $3$.
When we run Algorithm~\ref{alg:framework} with $r=3,4$,
it fails to give a desired tensor decomposition.
So, we apply Algorithm~\ref{alg:framework} with $r=5$
and get the symmetric $X$-decomposition with
\[
\widetilde{V} =  \left[ \begin {array}{cc}  3.029- 0.07505\,i&- 6.079+ 0.2073\,i
\\ \noalign{\medskip}- 1.047+ 0.1114\,i& 0.2338+ 0.2814\,i
\\ \noalign{\medskip} 2.615- 0.09868\,i& 4.970- 0.2555\,i
\\ \noalign{\medskip} 7.927- 0.008981\,i& 23.68- 0.03874\,i
\\ \noalign{\medskip}- 5.685- 3.079\,i& 10.84- 10.80\,i\end {array}
 \right],\quad \widetilde{\Lambda} =  \left[ \begin {array}{c} - 0.9769-
 0.06497\,i\\ \noalign{\medskip}- 1.4650-
 0.1107\,i\\ \noalign{\medskip} 3.3470+
 0.1587\,i\\ \noalign{\medskip} 2.0980+
 0.01425\,i\\ \noalign{\medskip}- 0.002438
+ 0.002715\,i\end {array} \right].
\]
We have $\lVert {\mA} \rVert = 41632.56$
and the error $\lVert \mA - \widetilde{\mA} \rVert = 4 \cdot 10^{-13}$.
% \quad \frac{\lVert \mA - \widetilde{\mA} \rVert }{\lVert \widetilde{\mA} \rVert}
% =  9\cdot 10^{-18}.
\end{example}

\begin{example}
\label{example:rationalscroll2}
Let $\PP X \subseteq \PP^3$ be the union of two planes defined by
$(x_3 - x_2)(x_1 - x_0) = 0$.
Let $\mA \in \S^3(\mathbb{C}^4)$ be the symmetric tensor such that
\begin{align*}
\mA_{3000} &= 2,\mA_{2100} = 1, \mA_{2010} = 5,\mA_{2001} = -3, \mA_{1200} = 5,\mA_{1110} = 10, \mA_{1020} =9, \\
\mA_{1101} &= 2,\mA_{1011} = 7, \mA_{1002} = 5,\mA_{0300} = -5,\mA_{0210} = 2, \mA_{0120} = 8, \mA_{0030} = 29, \\ \mA_{0201} &= -6,\mA_{0111} = 6,\mA_{0021} =  7, \mA_{0102} =4,\mA_{0012} = 5,\mA_{0003} = -9.
\end{align*}
The vanishing ideal of $Y \subseteq \mathbb{C}^3$ is
\[
I(Y) = \langle (y_3 - y_2)(y_1 - 1)  \rangle .
\]
The maximum rank of flattening matrices of $\mA$ is $4$.
When we run Algorithm~\ref{alg:framework} with $r=4$,
it fails to give a desired tensor decomposition.
So, we use $r = 5$ and apply Algorithm \ref{alg:framework}.
It returns the symmetric $X$-decomposition
\[
\widetilde{V} =  \left[ \begin {array}{ccc} - 2.0&- 1.0&- 1.0\\ \noalign{\medskip} 1.0
&- 1.487& 1.0\\ \noalign{\medskip} 1.0& 2.287& 1.0
\\ \noalign{\medskip}- 1.0& 1.0& 1.0\\ \noalign{\medskip} 1.0&
 0 %  -{ 1.596 \cdot 10^{-19}}
&- 2.0\end {array} \right] ,\quad \widetilde{\Lambda} =  \left[ \begin {array}{c}  1.0\\ \noalign{\medskip}-
 1.249\\ \noalign{\medskip} 2.249
\\ \noalign{\medskip}- 1.0\\ \noalign{\medskip} 1.0\end {array}
 \right].
\]
We have $\lVert \mA \rVert = 104.86$ and the error
$\lVert \mA  - \widetilde{\mA} \rVert = 10^{-15}$.
%\frac{\lVert \mA  - \widetilde{\mA} \rVert }{\lVert \widetilde{\mA} \rVert}
%=  10^{-17}.
\end{example}

\begin{example}
Let $\PP X\subseteq \PP^3$ be the surface defined by
$-3x_1x_2^2 + x_1^3 - x_0^2x_3 = 0$.
Then $Y \subseteq \mathbb{C}^3$ is the monkey saddle surface whose vanishing idea is
\[
I(Y) = \langle -3y_1y_2^2 + y_1^3 - y_3  \rangle .
\]
Let $\mA \in \S^3(\mathbb{C}^4)$
be the symmetric tensor such that
\begin{align*}
\mA_{3000} &= 5,\mA_{2100} = -1, \mA_{2010} = 6,\mA_{2001} = -13, \mA_{1200} = 9,\mA_{1110} = 8, \mA_{1020} =6, \\
\mA_{1101} &= -33,\mA_{1011} = -16, \mA_{1002} = 87,\mA_{0300} = 17,\mA_{0210} = 24, \mA_{0120} = 10, \mA_{0030} = 12, \\ \mA_{0201} &= -91,\mA_{0111} = -54,\mA_{0021} =  -38, \mA_{0102} = 233,\mA_{0012} = 5,\mA_{0003} = -739.
\end{align*}
The maximum rank of flattening matrices of $\mA$ is $4$.
When we run Algorithm~\ref{alg:framework} with $r=4,5$,
it fails to give a desired tensor decomposition.
So, we apply Algorithm \ref{alg:framework} with $r=6$
and get the symmetric $X$-decomposition
\tiny
\begin{align*}
\widetilde{V} &=  \left[ \begin {array}{ccc}  5.936- 0.8124\,i& 3.582- 0.4510\,i&-
 19.62+ 2.978\,i\\ \noalign{\medskip}- 0.4115- 0.5223\,i&- 1.801+
 1.185\,i& 9.226- 2.511\,i\\ \noalign{\medskip} 2.042- 0.01150\,i&-
 1.031+ 0.008333\,i& 2.008- 0.001995\,i\\ \noalign{\medskip}- 0.6422-
 0.02100\,i& 0.9440+ 0.001783\,i& 1.453+ 0.03664\,i
\\ \noalign{\medskip} 2.157- 0.08836\,i& 1.576- 0.04318\,i&- 6.043+
 0.3063\,i\\ \noalign{\medskip}- 1.0-{ 1.505\cdot 10^{-18}}\,i&-{
 3.163\cdot 10^{-18}}-{ 1.004\cdot 10^{-18}}\,i&- 1.0+{ 1.519\cdot
10^{-18}}\,i\end {array} \right],\\
\widetilde{\Lambda} &=  \left[ \begin {array}{c}  0.03454+
 0.02496\,i\\ \noalign{\medskip} 0.001235+
 0.00440\,i\\ \noalign{\medskip}- 0.9564-
 0.01231\,i\\ \noalign{\medskip} 2.0480-
 0.05064\,i\\ \noalign{\medskip} 1.8730+
 0.03358\,i\\ \noalign{\medskip} 2.0+{
 1.8560\cdot 10^{-17}}\,i\end {array} \right],\quad \lVert \mA \rVert = 1275.93,\quad \lVert \mA  - \widetilde{\mA} \rVert = 2\cdot 10^{-15}, \quad \frac{\lVert \mA  - \widetilde{\mA} \rVert }{\lVert \widetilde{\mA} \rVert} =  2\cdot 10^{-18}.
\end{align*}
\end{example}

In the next two examples, we still display the tensor entries
$\mA_\alpha$ according to the lexicographic order,
but we drop the labelling indices, for cleanness of the paper.
\begin{example}
Let $\PP X \subseteq \PP^4$ be the curve defined by
\begin{align*}
x_1x_3 - x_0 x_2 - x_0x_1 + x_0 x_3 &= 0,\,
x_3^2 - x_0x_1 - x_0^2 =0,\\
x_1x_4 + 4 x_1x_3 - x_1 x_2 - x_1^2 + 5x_1 &=0.
\end{align*}
The vanishing ideal of $Y\subseteq \mathbb{C}^4$ is then
\[
I(Y) = \langle  y_1y_3 - y_2 - y_2 + y_3,  y_3^2 - y_1 -1, y_1y_4 + 4y_1y_3 - y_1y_2 - y_1^2 + 5y_1\rangle .
\]
Let $\mA \in \S^3(\mathbb{C}^5)$ be the symmetric tensor whose
%%%%%%%%%%%%%%%%%%%%%%
\iffalse
\begin{align*}
\mA_{30000} &= -7,\mA_{21000} = -2, \mA_{20100} = 87, \mA_{20010} = 25, \mA_{20001} = 20,\mA_{12000} = 26,\\
\mA_{11100} &=334,\mA_{10200} = -233,\mA_{11010} = 60,\mA_{10110} = -45,\mA_{10020} = -9, \mA_{11001} = 130, \\
\mA_{10101} &= -144, \mA_{10011} = -74,\mA_{10002} = 182,\mA_{03000} = 406,\mA_{02100} =  1754, \mA_{01200} = 1150, \\
\mA_{00300} &= 13647,\mA_{02010} = 300, \mA_{01110} = 156, \mA_{00210} = 2353,\mA_{01020} = 24, \mA_{00120} = 421,\\
\mA_{00030} &= 85,
\mA_{02001} = 830,\mA_{01101} = 610, \mA_{00201} = 6500,\mA_{01011} = 60,\mA_{00111} = 1050,\\ \mA_{00021} &= 150,
\mA_{01002} =550,\mA_{00102}= 3630,\mA_{00012} = 880,\mA_{00003} = -250.
\end{align*}

\fi
%%%%%%%%%%%%%%%%%%%%%%%%
$35$ entries are
\begin{multline*}
-7,  -2,   87,   25,   20,  26,
334,  -233,  60,  -45,  -9,   130,
-144,   -74,  182,  406,   1754,   1150,   \\
13647,  300,   156,   2353,  24,   421,
 85,  830,  610,   6500,  60, 1050,  150,
 550,  3630,  880,  -250.
\end{multline*}
The maximum rank of flattening matrices of $\mA$ is $5$.
When we run Algorithm~\ref{alg:framework} with $r=5$,
it fails to give a desired tensor decomposition.
So, we apply Algorithm \ref{alg:framework} with $r=6$
and get the symmetric $X$-decomposition
\[
\widetilde{V} = \left[ \begin {array}{cccc}
680.9 & 17130.0  & 26.11  & 17700.0  \\
 7.717  & 18.02 &  2.952 & 8.926 \\
0 & -1.0 & -1.0 & -10.52- 8.239\,i  \\
-0.2272 &- 0.4521 &- 0.8791 & -2.163  \\
0.7526  & 1.568 & 1.324 & -7.975 \\
2.945 & -10.78 & -1.986 & -4.891
\end {array} \right] , \quad
\widetilde{\Lambda}  =  \left[ \begin{array}{c}
0  \\  1.1720  \\  0 \\   -6.780  \\   3.871 \\  -5.263
\end {array} \right].
\]
We have the norm
$\lVert {\mA} \rVert = 29222.16$
and the error
$\lVert \mA - \widetilde{\mA} \rVert =  10^{-10}$.
%
% \frac{\lVert \mA - \widetilde{\mA} \rVert }{\lVert \widetilde{\mA} \rVert}
% = 3\cdot 10^{-15}.
%
\end{example}

\begin{example}
Let $\PP X \subseteq \PP^4$ be the surface defined by
\begin{align*}
x_3^2 + x_4^2 - x_0x_1 = 0,\quad x_3 x_4 - x_0 x_2 =0.
\end{align*}
Then the vanishing ideal of the variety $Y\subseteq \mathbb{C}^4$ is
\[
I(Y) = \langle y_3^2 + y_4^2 - y_1,  y_3 y_4 - y_2\rangle .
\]
Let $\mA \in \S^4(\mathbb{C}^5)$ be the symmetric tensor whose
%%%%%%%%%%%%%%%%%%%%%
\iffalse

\begin{align*}
\mA_{40000} &= 22,\mA_{31000} =38,\mA_{30100} =  89,\mA_{30010} = 6, \mA_{30001} =34,
\mA_{22000} = 220, \\
\mA_{21100} &= 490, \mA_{20200} = 79, \mA_{21010} = 119,\mA_{20110} = 65, \mA_{20020} = 32, \mA_{21001} = 165, \\
\mA_{20101} &= 71,
\mA_{20011} = 89, \mA_{20002}=6,
\mA_{13000} = 2216,\mA_{12100} =3044, \mA_{11200} = 686, \\
\mA_{10300} &= 653, \mA_{12010} = 1029,
\mA_{11110} =490, \mA_{10210} = 239,\mA_{11020} = 195,\mA_{10120} = 173, \\
\mA_{10030} &= 48,
\mA_{12001} = 1111,\mA_{11101} = 574, \mA_{10201} = 257,\mA_{11011} = 490,\mA_{10111} = 79,\\
\mA_{10021} &= 65, \mA_{11002} = 25,\mA_{10102} = 317,\mA_{10012} = 71, \mA_{10003} = 100, \mA_{04000} = 21424,\\
\mA_{03100} &= 20440, \mA_{02200} = 6028,\mA_{01300} =4570,
\mA_{00400} =  1615, \mA_{03010} =8033, \mA_{02110} = 3788,\\
\mA_{01210} &= 1918,
\mA_{00310} = 929,\mA_{02020} = 1553
\mA_{01120} = 1162,\mA_{00220} = 415,\mA_{01030} =455, \\
\mA_{00130} &= 233, \mA_{00040} = 116, \mA_{03001} =8187, \mA_{02101} = 4316,
\mA_{01201} =1954, \mA_{00301} = 1007,\\
\mA_{02011} &= 3044, \mA_{01111} =686,
\mA_{00211} = 653, \mA_{01021} =490,  \mA_{00121} = 239, \mA_{00031} = 173, \\
\mA_{02002} &= 663,
\mA_{01102} = 1882, \mA_{00202} = 271,\mA_{01012} = 574,\mA_{00112} = 257, \mA_{00022} = 79,\\
\mA_{01003} &= 621,\mA_{00103} = 335,\mA_{00013} = 317,\mA_{00004} = -54.
\end{align*}

\fi
%%%%%%%%%%%%%%%%%%%%%%
$70$ entries are respectively
\begin{multline*}
22,38,  89, 6, 34, 220, 490, 79,  119, 65, 32,  165,
71, 89, 6, 2216,3044,  686, 653, 1029, \\
490, 239,195, 173, 48, 1111,574,  257, 490, 79,
65, 25, 317, 71,  100,  21424, 20440, \\
6028,4570, 1615, 8033,  3788,
1918, 929,1553, 1162,415,455,  233,  116, 8187, 4316, 1954, \\
1007, 3044, 686, 653, 490,  239,  173,
663, 1882, 271, 574, 257,  79,
621,335, 317, -54.
\end{multline*}
The maximum rank of flattening matrices of $\mA$ is $10$.
When we run Algorithm~\ref{alg:framework} with $r=10$,
we get the symmetric $X$-decomposition
\tiny
\begin{align*}
\widetilde{V} &=  \left[ \begin {array}{cccc}
5.0 & -2.0 & -1.0 & 2.0 \\
0.7596+ 0.2348\,i & 0.2187+ 0.5369\,i & 1.0 & 0.2187+ 0.5369\,i \\
1.937+ 0.1658\,i & 0.9718+ 0.08532\,i & 1.0 & 0.9718+ 0.08532\,i\\
2.0  & 1.0 &- 1.0 & -1.0  \\
 1.0 & 0 & 0 & 1.0  \\
 2.065+ 0.07655\,i & -1.033-0.03706\,i & 1.0 &- 1.033- 0.03706\,i\\
 5.003+ 0.006546\,i & -2.001-0.001636\,i & 1.0 &- 2.001- 0.001636\,i \\
 2.0 & -1.0 & -1.0 & 1.0 \\
 4.999+ 0.01109\,i& 2.0+ 0.002774\,i & 1.0 & 2.0+ 0.002774\,i \\
 8.0 & 4.0 & 2.0  & 2.0
\end {array} \right], \\
\widetilde{\Lambda} &=  \left[ \begin {array}{c}
- 10.0 \\  0.505+ 0.297\,i \\   3.232- 0.413\,i
\\   11.0  \\  11.0 \\ -3.778 + 0.132\,i
\\  -7.970+ 0.053\,i \\  5.0 \\  6.01- 0.069\,i \\  7.0
\end {array} \right],\quad \lVert {\mA} \rVert = 147394.37
,\quad \lVert \mA - \widetilde{\mA} \rVert =  8\cdot 10^{-11}.
%
%\quad \frac{\lVert \mA - \widetilde{\mA} \rVert }
%{\lVert \widetilde{\mA} \rVert} = 6 \cdot 10^{-16}.
%
\end{align*}
\end{example}

%
%\subsection{Vandermonde rank decomposition of non-symmetric tensors}\label{subsec:Vandermonde}
%
We conclude this section by considering various examples on Segre varieties.
\begin{example}
(Vandermonde decompositions of nonsymmetric tensors)
Each tensor $\mA \in (\cpx^{d+1})^{\otimes k}$ has a
Vandermonde decomposition as in \reff{eqn:Vanderdecomp},
which is proved in Theorem~\ref{thm:HDnst}.
We can view $\mA$ as a tensor in $\S^d(\cpx^{2^k})$
with the set $X \subseteq \cpx^{2^k}$ such that
$\PP X = \PP^1 \times \cdots \times \PP^1$ ($\PP^1$ is repeated $k$ times).
Let \[ n = 2^k - 1. \]
A vector $x \in \cpx^{2^k}$ can be labelled as
$x = (x_\nu)$, with binary vectors $\nu \in \{0, 1\}^k$.
Under this labelling, the set $X$ is defined by the homogeneous equations 
(c.f.~\cite[Example~2.11]{Harris1992}) or \cite[Section~4.3.5]{Land12})
\be \label{x:mu+nu=et+th}
x_{\mu} x_\nu - x_\eta x_\theta = 0
\ee
for all $\mu, \nu, \eta, \theta \in \{0, 1\}^k$ such that
for some $1 \leq i < j \leq k$,
\[
\mu_l  = \nu_l = \eta_l = \theta_l \, (l\ne i,j), \quad
\mu_i+\nu_j = \eta_j + \theta_i.
\]
Under the dehomogenization $x_{0\ldots 0}=1$,
the corresponding affine variety $Y\subseteq \mathbb{C}^{2^k-1}$
consists of vectors $y$, labelled as
$y = (x_\mu)$ with $0 \ne \mu \in \{0, 1\}^k$,
%satisfying similar equations as in \reff{x:mu+nu=et+th}.
defined by the vanishing ideal
\[
I(Y) = \langle  y_{\mu}y_{\nu} - y_{\eta} y_{\theta} \rangle,
\]
where $\mu,\nu,\eta,\theta\in \{0,1\}^k$ are as above and $y_{0 \ldots 0} = 1$.

We apply Algorithm \ref{alg:framework} to compute Vandermonde decompositions
for random $\mA \in (\cpx^{d+1})^{\otimes k}$ whose entries
are randomly generated (obeying the normal distribution).
%The value of $r$ is chosen to be the expected generic rank
%given by in Example~\ref{example:segre and hypersurface}.
For all the instances, Algorithm~\ref{alg:framework} successfully got
Vandermonde rank decompositions.
The relative errors $\frac{\lVert \mA - \widetilde{\mA} \rVert }{\lVert \mA \rVert}$
are all in the magnitude of $O(10^{-16})$.
The consumed computtaional time (in seconds) is also reported.
The results are displayed in Table~\ref{table:decompsegre}.
%\red{**Can Ke do more computations to fill the table?}
%%%%%%%%%%%%%%%%%%%%%%%%%%%%%%
%%% the old table

\begin{table}[htb]
\centering
\caption{Computational results on symmetric $X$-decompositions on Segre varieties}
 \label{table:decompsegre}
\begin{tabular}{lllrr | lllrr}
\toprule
$k$  & $n$ & $d$ & $r$  & time & $k$  & $n$ & $d$ & $r$  & time \\ \hline
2 & 3 & 3 & 4 &  7.07   & 3 & 7 & 3 & 7  &  7.45  \\
2 & 3 & 4 & 4 &  7.08   & 3 & 7 & 3 & 9  &  9.74 \\
2 & 3 & 4 & 5 &  7.11   & 3 & 7 & 3 &10  &  19.04\\
2 & 3 & 4 & 6 &  7.31   & 3 & 7 & 4 & 8  &  6.85 \\
2 & 3 & 4 & 7 &  7.45   & 3 & 7 & 4 & 9  &  6.88 \\
2 & 3 & 4 & 8 &  8.11   & 3 & 7 & 4 & 10 &  6.54  \\
2 & 3 & 5 & 5 &  7.10   & 3 & 7 & 4 & 11 &  8.35  \\
2 & 3 & 5 & 6 &  7.08   & 3 & 7 & 4 & 12 &  11.37  \\
2 & 3 & 5 & 7 &  7.22   & 3 & 7 & 4 & 13 &  22.84 \\
2 & 3 & 5 & 8 &  7.17   & 3 & 7 & 4 & 14 &  41.45 \\
2 & 3 & 5 & 9 &  7.37   & 3 & 7 & 4 & 15 &  69.26 \\
2 & 3 & 5 & 10 & 7.76   & 3 & 7 & 4 & 16 &  134.47 \\
2 & 3 & 5 & 11 & 8.87   & 3 & 7 & 4 & 17 &  224.97  \\
2 & 3 & 6 &12  & 7.84   & 3 & 7 & 4 & 18 &  382.00\\
2 & 3 & 6 &13  & 7.97   & 3 & 7 & 5 & 19 &  8.35 \\
2 & 3 & 6 & 14 & 8.50   & 3 & 7 & 5 & 20 &  8.45 \\
2 & 3 & 6 & 15 & 9.26   & 3 & 7 & 5 & 21 &  8.65 \\
2 & 3 & 6 & 16 &12.33   & 3 & 7 & 5 & 22 &  8.65  \\
2 & 3 & 7 & 17 & 8.27   & 3 & 7 & 5 & 23 &  8.80 \\
2 & 3 & 7 & 18 &12.11   & 3 & 7 & 5 & 24 &  9.08  \\
2 & 3 & 7 & 19 &34.24   & 3 & 7 & 5 & 25 &  8.96   \\
2 & 3 & 7 & 20 &722.98  & 3 & 7 & 6 & 26 &  9.10   \\
\bottomrule
\end{tabular}
\end{table}

%%%%%%%%%%%%%%%%%%%%%%%%%%%%%%
%\begin{table}[htb]
%\centering
%\caption{decompositions on Segre varieties}
% \label{table:decompsegre}
%\begin{tabular}{llll|llll|llll} \hline
%$k$  & $n$ & $d$ & time & $k$  & $n$ & $d$ & time & $k$  & $n$ & $d$ & time   \\ \hline
%2 & 3 & 3 &   7.07  & 3 & 7 & 3 &   7.45 & 4 & 15 & 3 &   \\ \hline
%2 & 3 & 4 &   8.11  & 3 & 7 & 4 &   6.54 & 4 & 15 & 4 &    \\ \hline
%2 & 3 & 5 &   8.87  & 3 & 7 & 5 & 224.97 & 4 & 15 & 5 &    \\ \hline
%2 & 3 & 6 &   12.33 & 3 & 7 & 6 &   8.65 & 4 & 15 & 6 &    \\ \hline
%2 & 3 & 7 &   8.27  & 3 & 7 & 7 &   8.80 & 4 & 15 & 7 &    \\ \hline
%2 & 3 & 8 &  722.98 & 3 & 7 & 8 &   9.10 & 4 & 15 & 8 &     \\ \hline
%\end{tabular}
%\end{table}

\end{example}

\section{Conclusion}
In this paper, we discuss how to compute symmetric $X$-decompositions
of symmetric tensors on a given variety $X$.
The tool of generating polynomial is used to do the computation.
Based on that, give an algorithm for computing symmetric $X$-decompositions.
Various examples are given to demonstrate the correctness
and efficiency of the proposed method.

%\subsection*{Acknowledgments}
%%The authors thank the anonymous reviewer for his careful reading and numerous helpful suggestions that greatly improved this article.

%KY's work is supported by NSFC no.~11688101, NSFC no.~11801548, National Key R\&D Program of China Grant no.~2018YFA0306702, the Hundred Talents Program of the Chinese Academy of Sciences as well as the recruitment program for young professionals of China.

%
%\bibliography{mybib}
\bibliographystyle{plain}
\def\cdprime{$''$} \def\cprime{$'$}
\def\cprime{$'$} \def\cprime{$'$}
\def\Dbar{\leavevmode\lower.6ex\hbox to 0pt{\hskip-.23ex \accent"16\hss}D}
\def\cprime{$'$}

\end{document}